\documentclass{article}
\usepackage{graphicx} % Required for inserting images
\usepackage{amsmath,amsfonts,amsthm,hyperref}
\usepackage[margin=3 cm,bottom=35mm,footskip=15mm,top=35mm]{geometry}

\usepackage{subcaption}
\usepackage{array}
\usepackage{multicol}
\usepackage{booktabs}

\usepackage{tikz}

\newtheorem{theorem}{Theorem}[section]

\newtheorem{proposition}[theorem]{Proposition}

\newtheorem{question}[theorem]{Question}

\newcommand{\jordan}[1]{{\color{violet} ($\spadesuit$ Jordan: #1)}}

\title{PatternBoost: Constructions in Mathematics with a Little Help from AI}

\author{
Fran\c cois Charton\thanks{FAIR, Meta -- CERMICS, Ecole des Ponts. Email:
		\href{mailto:fcharton@meta.com} {\nolinkurl {fcharton@meta.com}}.}
  \and
	Jordan Ellenberg\thanks{Department of Mathematics, University of Wisconsin–Madison. Email:
          \href{mailto:ellenber@math.wisc.edu} {\nolinkurl {ellenber@math.wisc.edu}}. Research supported in part by NSF Grant DMS-2301386.}
        \and
  Adam Zsolt Wagner\thanks{Mathematical Sciences Department, Worcester Polytechnic Institute. Email:
		\href{mailto:zadam@wpi.edu} {\nolinkurl {zadam@wpi.edu}}.} 
 \and
	Geordie Williamson\thanks{Sydney Mathematical Research Institute, University of Sydney. Email:
          \href{mailto:g.williamson@sydney.edu.au } {\nolinkurl{g.williamson@sydney.edu.au }}
            Research supported in part by Australian Research Council Discovery Project DP230102982,
            an Australian Research Council Laureate Fellowship FL230100256 and the Max Planck Humboldt
            Research Award.}
}

\begin{document}

\maketitle

\begin{abstract}

    We introduce PatternBoost, a flexible method for finding interesting constructions in mathematics. Our algorithm alternates between two phases. In the first ``local'' phase, a classical search algorithm is used to produce many desirable constructions. In the second ``global'' phase, a transformer neural network is trained on the best such constructions. Samples from the trained transformer are then used as seeds for the first phase, and the process is repeated. We give a detailed introduction to this technique, and discuss the results of its application to several problems in extremal combinatorics. The performance of PatternBoost varies across different problems, but there are many situations where its performance is quite impressive. Using our technique, we find the best known solutions to several long-standing problems, including the construction of a counterexample to a conjecture that had remained open for 30 years.
\end{abstract}

\section{Introduction}

In this paper we introduce and give examples of a computational method, called PatternBoost, for generating interesting constructions in pure mathematics. 

The method involves an iterative alternation between ``local'' and ``global".  The former, as we will see, is typically a simple greedy algorithm.  The latter is a form of genetic algorithm using a {\em transformer}, a flexible machine learning technique which we believe is especially well-suited for problems of this nature.

To get the sense of how this iteration might look, consider an example of a collective human endeavor, the design of bicycles.  The ``local" step involves many individual bicyclemakers each making a sequence of careful tweaks to their design, trying to make it as efficient and comfortable as possible.  The ``global" step involves people living in the world, seeing lots of bicycles around, each one of which has been carefully optimized locally, and then developing a new bicycle design based on those they have observed.  Of course this new design would then be carefully optimized by both its designer and others; and if after those tweaks, the new bikes turned out to be notably comfortable and efficient, a lot of them would be sold, and they would join the general fleet of bicycles observed by the next prospective designer,... and on and on we go.

Mathematical objects are not bicycles. But human beings can abstract features of bicycles and develop new objects that we recognize as bicycles, despite their not being identical with any existing examples, and mathematicians do the same with mathematical objects. However, it is typically difficult to automate this process. Our hope for the methods described here is that techniques from machine learning (and in particular transformers) have at least some capabilities of this kind -- that presented with a list of mathematical entities, they can produce further examples which are, at least in some respects, ``of the same kind" as those observed.

Our work is strongly influenced by earlier work \cite{wagner2021constructions} of the third author. There the cross-entropy method in reinforcement learning is used to find counter-examples to several problems in combinatorics. The problem with the cross-entropy method (in the bare form used in \cite{wagner2021constructions}) is its scaling: the vanilla neural network becomes difficult to train when the sequence length exceeds a few hundred tokens. A very similar problem is encountered in AI when trying to use a vanilla neural network for next token prediction on long sequences of letters or words, and it is on this style of problem that the transformer architecture excels.

One of the main strengths of PatternBoost is its broad applicability. By adding a global step that uses transformers to suggest better starting points for the local search, PatternBoost can improve results across many optimization problems. One may think of PatternBoost as an extra layer that can be placed on top of any local search method, often leading to better solutions than with local search alone. Simply put, whatever local search algorithm we have, PatternBoost can often make it better.

We emphasize that our chief aim is to develop a useful and simple tool for working mathematicians, one which does not require extensive expertise in machine learning or access to industry-level computing power. One difficulty for using machine learning as a practical tool in mathematics is that machine learning is hard! One can lose many hours tuning hyperparameters, exploring different tokenization schemes, and so on.  A virtue of PatternBoost, as we see it, is that the transformer architecture appears very resilient, and can often be used ``off the shelf'' without much tinkering by the mathematician whose expertise and interest may be elsewhere.  (Though as we will see in the text, these choices, e.g.~the choice of tokenization, can actually make a difference in performance!) We use a beautiful and simple implementation of a transformer, given by Andrej Kaparthy's Makemore, which in our experiments seems to produce useful output over a wide range of mathematical contexts.

The problems discussed here are just the first few we tried as we developed PatternBoost -- we hope and expect that other mathematicians will find it fun to carry out further experiments, helping to map out the still-mysterious territory of which mathematical problems are amenable to ML-boosted approaches.  In particular: the examples discussed in this paper are mostly in the realm of extremal combinatorics, where the transformer is used to construct combinatorial examples that are as large as possible subject to some constraints.  Certainly combinatorial structures are the entities that are easiest to present as input to the transformer; but we do not see the method as in principle limited to that part of mathematics.  Indeed, there is nothing in the method which is specific to mathematics at all!  We would be interested to understand whether there are problems outside mathematics to which PatternBoost can be applied.  One obvious challenge is that a proposed example in mathematics can often be evaluated mechanically, reliably, and quickly, and this is crucial for PatternBoost; in other domains, evaluation may pose more difficulties.

The code accompanying this paper can be found at:

\noindent\url{https://github.com/zawagner22/transformers_math_experiments}.

\iffalse

\jordan{I have tried to incorporate Geordie's bullet points into the introduction above; I think I got everything in except sonnet probability.}

\jordan{OK I tried to write something short that hit the main bullet points below!}

todo. 
\begin{itemize}
    \item this is a blackbox that we can add to any local search method, to get an improvement
    \item local-global method: transformer understands the global structure, local search fixes small local mistakes, it seems good to iterate these two
    \item humans don't draw a graph edge by edge, they have a global picture
    \item a new spin on genetic algorithms
    \item useful and simple tool for mathematicians
    \item please use our methods, it's fun!
\end{itemize}
\fi

\subsection*{Related works}

\emph{Neural architectures for optimization} were introduced in the 1980s. Hopfield and Tank~\cite{hopfield_tank} demonstrated their use for the Travelling Salesman Problem, and Dhingra and Rao~\cite{dhingra} for nonlinear problems. This is still an active field of research~\cite{amos2021optnetdifferentiableoptimizationlayer,chen2022neuraloptimizationmachineneural,amos2023tutorialamortizedoptimization}. These approaches leverage the Universal Approximation Theorem~\cite{cybenko1989approximation,hornik1991approximation}, which states that a trained neural network can approximate any continuous mapping. The neural network is trained to approximate the solution of a dynamical system, by minimizing the discrepancy between its output and the correct solution (see~\cite{brockett1991dynamical} for a general discussion of these methods, and~\cite{chen2019neuralordinarydifferentialequations} for applications to differential equations).  

\emph{Transformer architectures} were proposed by Vaswani et al.~\cite{transformer17}. In the original architecture, designed for machine translation, the model used two transformer stacks: a bidirectional encoder, tasked to process the input sequence, and an auto-regressive decoder, which would predict the translation, one token at a time, from its previous output and the input sequence transformed by the encoder. Makemore uses a simpler architecture, known as decoder-only, introduced by Radford et al. for GPT-2~\cite{radford2019language}, and which serves as the basis of modern Large Language Models~\cite{GPT3brown2020,llamatouvron2023}. 

Many authors have studied the use of \emph{Transformers in mathematics}. Exact integer arithmetic has proven to be surprisingly difficult~\cite{nogueira2021investigating,lee2023teachingarithmeticsmalltransformers}. On the other hand, transformers could be trained to perform harder tasks, such as symbolic integration~\cite{lample2019deep}, computing eigenvalues~\cite{charton2022linear}, and even solving problems of geometry from the International Mathematical Olympiads~\cite{trinh2024}. 

Finally, several works have applied neural architectures to \emph{actual research problems}: Davies et al.~\cite{davies2021advancing} to knot and representation theory, Wagner~\cite{wagner2021constructions} to graph conjectures, Romera-Paredes et al.~\cite{funsearch2023} to combinatorics, and Alfarano et al.~\cite{alfarano2024globallyapunovfunctionslongstanding} to control theory.

\section{A simple illustrative problem: many edges with no triangles}\label{sec:trifree}

  How many edges can an $n$-vertex graph have, if no three edges form a triangle? It is very interesting to introspect how we might go about solving this problem. A first step is probably to come up with some examples of graphs on a small number of vertices with many edges and no triangles:
  \[
\includegraphics[scale=0.3]{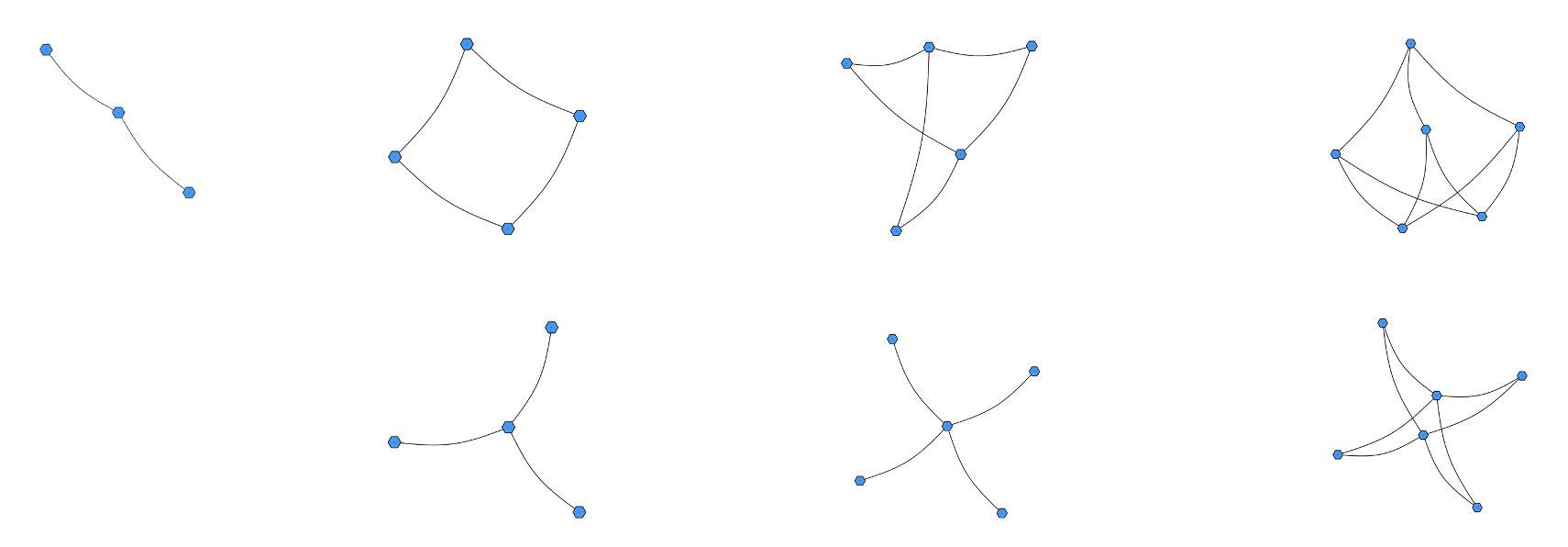}
    \]
We then might be lucky enough to notice that many/all examples in our list are in fact bipartite:
  \[
\includegraphics[scale=0.3]{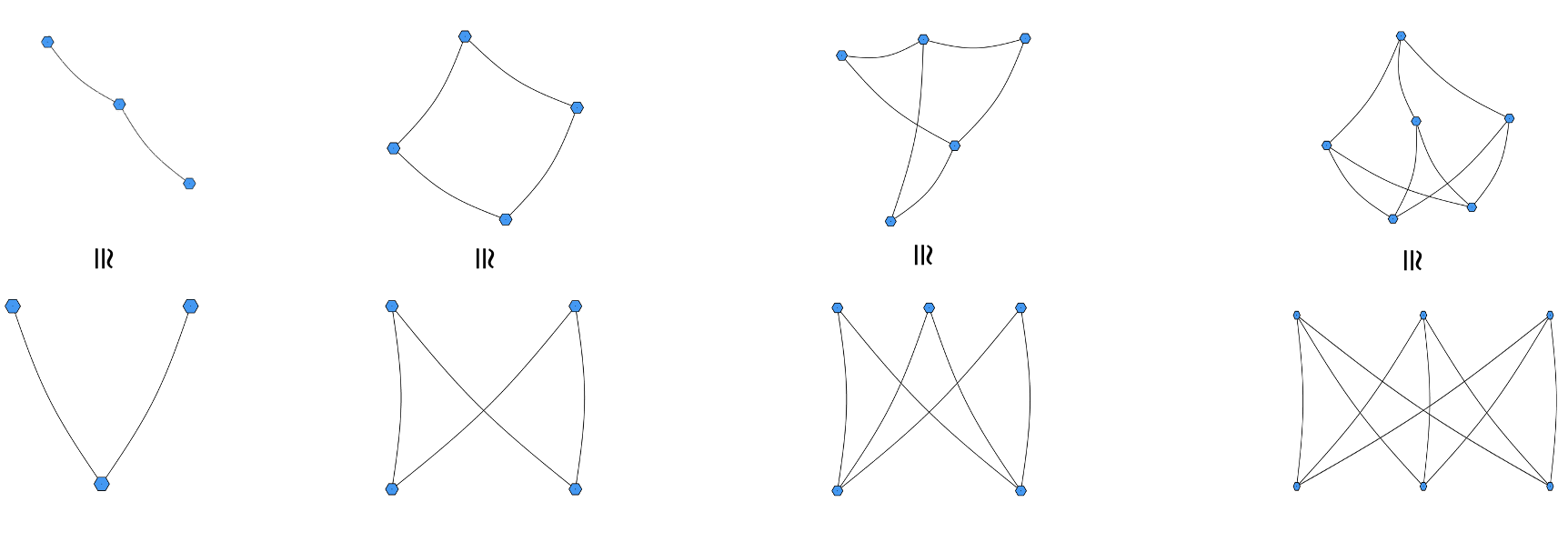}
\]
From here it is not difficult to guess that bipartite graphs provide the optimum. This is indeed the case, a fact known as Tur\'an's theorem for triangles, or Mantel's theorem.\footnote{The question of how one arrives at a \emph{proof} of this theorem is also very interesting, but will not be considered here. See \cite{gowers} for a very interesting discussion of how one might arrive at a proof on a computer.} Note a fascinating and important feature of problems of this type: some global constraint (bipartite) emerges from some local condition (no triangles).

Imagine for a moment that we don't know this theorem, and after some thought failed to come up on our own with an intuition about what graphs are likely best.  This accurately describes the scenario for most problems of this kind, where optimal constructions are not as easy to describe as they are in the case of graphs with no triangles.

We would like to have a general method that can help discover or at least approach these structures more or less by itself.  PatternBoost is such a method.  To get started, we need to come up with two things:
\begin{itemize}
    \item a local search method, and
    \item a score function.
\end{itemize}

The \emph{score function} should be a measure to evaluate how good our constructions are. In the present example, if someone gives us a graph (which may or may not contain triangles), we want to be able to reply to them with a number, where a higher number means that their graph was ``better'' for our problem.  Already at this stage there are choices to make; there is no canonical score function attached to a problem.  We could, for instance, decide to give a score of minus infinity if the graph contains any triangles at all, and return the number of edges otherwise. Another option, slightly less constrained, would be to score the graph $G$ with the function
$$\text{score}(G) = \#\text{edges}(G) - 2 \cdot \#\text{triangles}(G).$$
Often this second score function, which allows for the occasional mistake, leads to slightly better results. Note that if the graph has triangles in it, we can delete any edge from a triangle to make the score increase by (at least) one. Therefore, in both scoring functions, the constructions with the highest possible score  will automatically be triangle-free.

The \emph{local search method} is an algorithm that takes as input a graph which may or may not contain triangles, and outputs a graph whose score is at least as good as the score of the input. Typically, such algorithms proceed by modifying the graph in minor ways, for example by adding and removing edges, attempting to improve the score. This is why they are termed ``local search''. In many cases the local search is not able to improve the score; that is not a problem, as long as it can improve the input when there is an ``obvious'' way to do so.

In this paper, we are going to resist trying to optimize the local search, and instead stick with something very simple, since our primary goal is to illustrate the usefulness of the local-global iteration.  Our local search algorithm will have two phases:
\begin{enumerate}
    \item While there are still triangles in the graph, keep randomly deleting an edge that is in the maximum number of triangles.
    \item Once there are no more triangles in the graph, keep adding random edges to the graph, without creating new triangles, for as long as we can.
\end{enumerate}
A simple ``down and up" mechanism of this kind will be the local search protocol for most of the problems we discuss in the paper.  We will show that PatternBoost does quite well on some problems even with this very simple approach to local search.

=

We now explain the steps of PatternBoost in more detail in the context of the triangle-free graph problem.

\subsection*{Step 1: Create a starter database}

Let us set $n$, the number of vertices, to be 20 in this example. The first step is to create a dataset of graphs which will serve as our starting point.  The higher the scores the better, but any dataset will do. We proceed as follows: start with the empty graph, and run the above simple local search algorithm with this starting point (that is, add edges randomly for as long as we can, without creating triangles). We repeated this $40,000$ times, always starting from the empty graph, to get the score distribution as seen in Figure~\ref{fig:tri_results1}.  (Since the output of local search never has triangles, the score here is just the number of edges.)

The best score we were able to achieve this way was $99$, very close to the optimal value of $100$; this is achieved quite rarely, by two lucky runs out of the $40,000$ tries.  The bulk of the graphs form a smooth distribution peaking at $66$. We retain the top 25\% of scorers from this dataset; these graphs will serve as the training set for our transformer. The distribution of the scores of the training set can be seen in the histogram on the right-hand side of  Figure~\ref{fig:tri_results1}.

\begin{figure}[h!]
    \centering
    \includegraphics[width=0.45\linewidth]{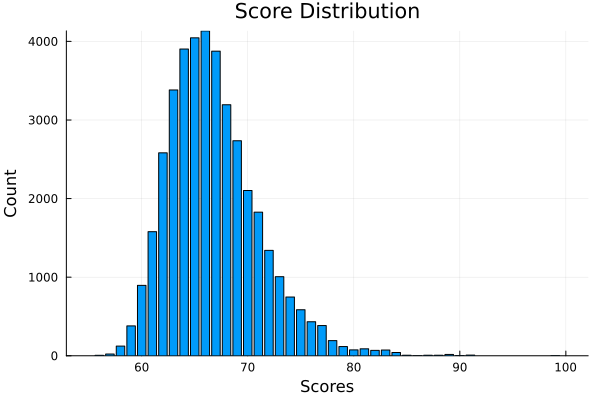}
    \includegraphics[width=0.45\linewidth]{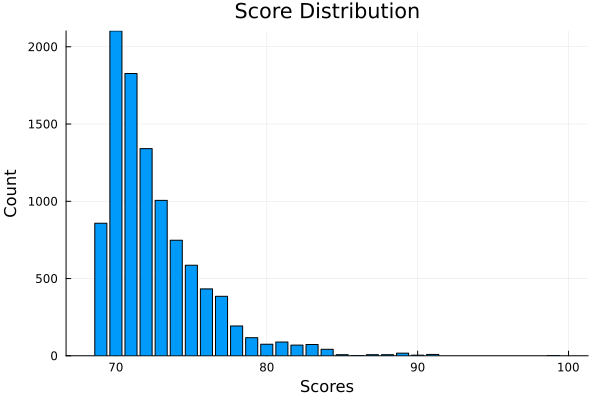}
    \caption{Left: score distribution of all the examples we generated by greedily adding edges for as long as possible. Right: the score distribution of the best 25\% of the examples from the left, which we will use as a training set.}
    \label{fig:tri_results1}
\end{figure}

Every graph in the training set can be represented by its adjacency matrix, which has $n^2=20^2=400$ entries. We can reduce it to $20\cdot 19/2=190$ by noticing that the adjacency matrix is symmetric and there are no loops, so we can use the upper triangular part of the matrix instead of the entire matrix. 

The transformer we're using takes one-dimensional inputs; so we can flatten this upper triangular matrix by writing it out row by row and putting a delimiter (in this case a comma) after every row, as in Figure~\ref{fig:tokenization_illustr}. We emphasize that this is far from the only way we could have presented this structure as a one-dimensional input.  For instance, we could have written the matrix diagonal by diagonal, in which case Figure~\ref{fig:tokenization_illustr} would have read $000,11,0$.  Choices like this really can make a difference in performance, and unfortunately, there is very little understanding at present of which choices are likely to lead to success. 
 In practice one tries everything, or everything one has time to try.

\begin{figure}
    \centering

\begin{tikzpicture}

% Matrix on the left
\node at (0, 0) {$\begin{matrix}
0 & 1 & 0 \\
  & 0 & 1 \\
  &   & 0 \\
\end{matrix}$};

% Arrow pointing to the right
\draw[->] (1.5,0) -- (3.5,0);

% String on the right
\node at (5, 0) {010,01,0};

\end{tikzpicture}
\caption{Flattening the adjacency matrix}
    \label{fig:tokenization_illustr}
\end{figure}

It is often helpful to further compress our data before starting the training process. We can achieve this by \emph{tokenizing} our data. This means that if we notice that the string ``00101'' occurs many times in our dataset, we introduce a new letter, say ``2'', to represent this string. An efficient way to do this is by using Byte-Pair Encoding tokenization, which finds frequently used strings in the dataset and agglomerates them into tokens automatically.  After using BPE with 100 tokens, the maximum word length reduces to 72. Note that not all tokens corresponds to strings of the same length. In our example, the string ``10'' received the name ``token 14'', whereas token 59 corresponded to ``0010000''.

\begin{table}[h!]
\begin{center}
\begin{tabular}{c c | c c}
 token ID & string & token ID & string  \\ 
 1 & ``0'' & 93 & ``0001101''\\  
 2 & ``1'' & 94  & ``11010000''\\
 3 & ``00'' &  95  & ``0100100'' \\
 4 & ``01'' & 96  & ``00010000'' \\
 5 & ``11'' & 97 & ``11010''  \\
 6 & ``100'' & 98 & ``010001'' \\
 7 & ``0100'' & 99 &``101000'' \\
\end{tabular}
\end{center}
\caption{Examples of some of the tokens created by BPE}
\label{tab:bpe_examples}
\end{table}

\subsection*{Step 2: Training the transformer}

We will use Makemore, a simple transformer implementation by Andrej Karpathy~\cite{makemore}. The advantage of his code is that it is openly available and easy to modify to suit our (or your!) purposes, and it gives a solid baseline that we can later try to beat with more sophisticated methods if we choose to do so. 

We used a tiny 2-layer transformer with 4 heads and embedding dimension 16 for this problem.  We train the transformer to produce sequences of tokens which are ``like" those in the initial dataset, in a fashion completely analogous to the means by which a transformer given a large database of English sentences (aka sequences of tokens, most of which are words) can be trained to produce more English sentences.  At every stage of training, the transformer can be asked to predict the next token following a given sequence of $k$ tokens.  In particular, for each $k$ and each graph $G$ in the dataset (represented as a sequence of tokens), we can ask the transformer to predict the $k+1$-st token given the first $k$; the ``loss" is a measure of how often the transformer fails to correctly predict the actual next token in $G$.

After training for 15,000 steps, the train loss went down to 2.07, and the test loss was 2.09.

\subsection*{Step 3: Obtain new constructions from the transformer}

Next, we asked the transformer to produce new constructions that are similar in some ``global'' sense to the best ones we have encountered so far, i.e. the ones in the training set. We produced $100,000$ new graphs this way in tokenized form. After decoding the sequences of tokens into matrices -- or trying to -- we found that $37,000$ had the correct number of entries (190) to be interpreted as the adjacency matrix of a $20$-vertex graph.  In other words, the majority of sequences produced cannot even be scored! This could be a sign that the transformer we used was too small, that the training set was too small so there was not enough data to figure out the meanings of all 100 tokens from context, or that we simply did not train for long enough.

\subsection*{Step 4: Run local search from the new constructions obtained from the transformer}
Let us take the $37,000$ valid constructions we received from our small model, and plug them back into our simple local search algorithm. That is, from each of these $37,000$ graphs, we first greedily delete edges to get rid of all triangles, and then add edges randomly for as long as possible without creating any new triangles. Our hope, and indeed the core idea of our method, is that the transformer has picked up something about the global structure of the best constructions.  This doesn't mean the constructions it produces will score more highly than those it was trained on; in its global view, it may certainly have made some small local mistakes when producing new ones. But it is reasonable to hope that the new constructions produced by the transformers will be much better starting points for the local search than the empty graphs we used in the first step.

And indeed, running the simple local search from the $37,000$ graphs produced by the transformer yields the score 99 a total of 47 times (compared to twice in the previous loop), and the maximum possible score of 100 a total of 46 times! Moreover, all graphs produced this way were different (but of course they are all isomorphic as there is a unique extremal graph up to isomorphism). Figure~\ref{fig:tri_phase2} shows how the score distribution has shifted after only one loop of the process.
\begin{figure}
    \centering
    \includegraphics[width=0.5\linewidth]{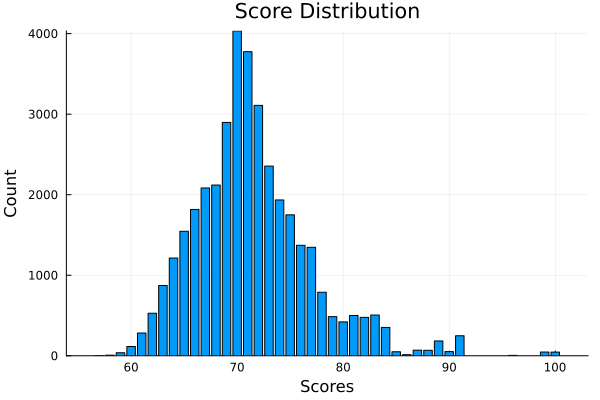}
    \caption{The score distribution in the second generation. The peak has shifted from 66 to 70.}
    \label{fig:tri_phase2}
\end{figure}

\subsection*{Step 5: Repeat this process}

An advantage of the method we described in this section is that it is trivial to iterate. We can repeatedly take the best $10,000$ of the previous generation, tokenize them with the same tokens we used before, and finetune our transformer on this new training set. Note that there is no need to restart the training from scratch every time. By doing 5 more loops, the model quickly learns to produce essentially only complete bipartite graphs, and the majority of these have equal part sizes,  see Figure~\ref{fig:tri_evolution}.  (Note that if the same constructions are created multiple times, they are not counted in these histograms.)

\begin{figure}
    \centering
    \includegraphics[width=0.3\linewidth]{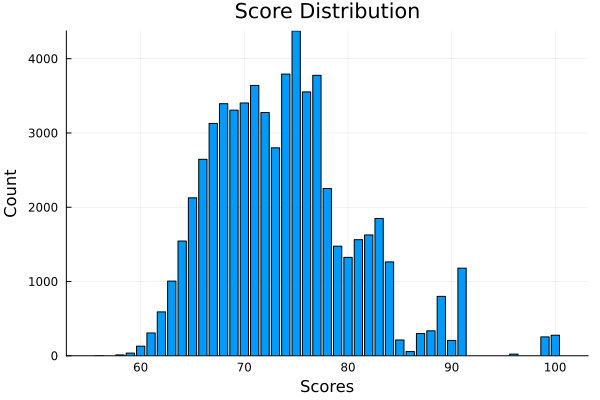}
    \includegraphics[width=0.3\linewidth]{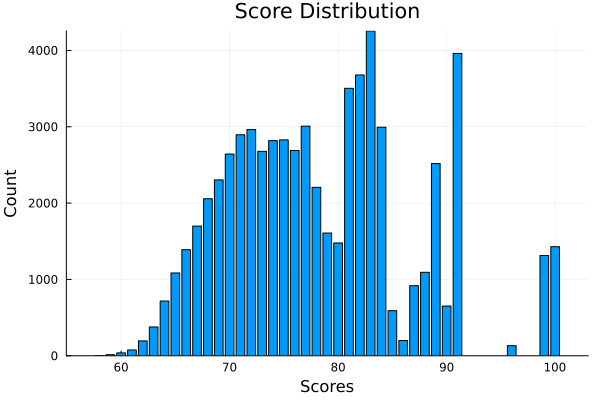}
    \includegraphics[width=0.3\linewidth]{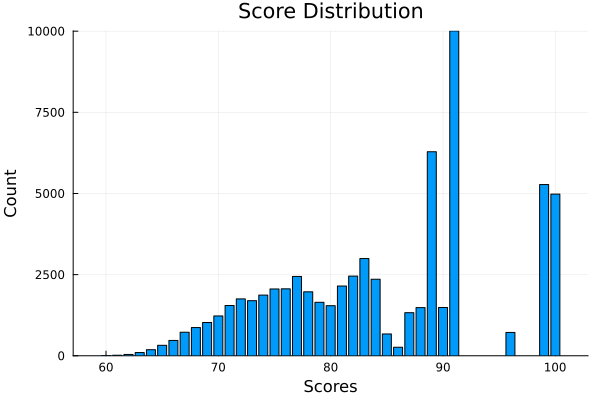}
    \includegraphics[width=0.3\linewidth]{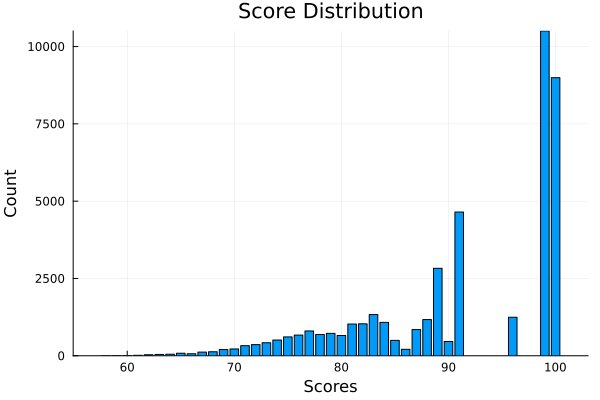}
    \includegraphics[width=0.3\linewidth]{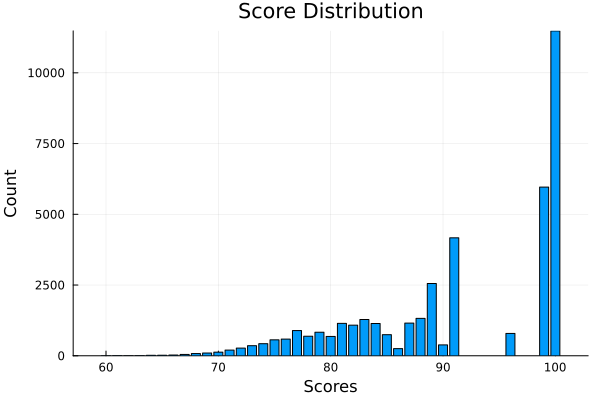}
    \caption{The evolution of scores across the next 5 generations. By the end, the model has learned to generate many different complete bipartite graphs.}
    \label{fig:tri_evolution}
\end{figure}

\subsection*{Discussion}
In this simple example we used a rather small transformer, and worked with tiny training sets. To discover more intricate patterns, as we will need to do in the other problems discussed in this paper, we will need to increase the sizes of both of these. A larger transformer is needed to be able to understand deeper structures, and a larger training set is then necessary for the training to work. This is partially to avoid overfitting, but also to have enough context to deduce the meanings of the tokens, and to have enough data for the transformer to spot global patterns between them and to avoid following ``false leads'' that may be present in small datasets by coincidence.

We have already seen that the PatternBoost alternation between local and global search outperforms local search alone, which most frequently arrives at triangle-free graphs with about 66 edges.  What about global search alone?  That is, what happens if we generate many random examples, but without improving them by local search at all, select the 1 percent of graphs which score best, pass those back to Makemore, and iterate that process?  This turns out to be disastrously bad; the largest triangle-free graph it finds, after fifty loops, has only 48 edges.  (Though better performance from a global-only approach was seen in \cite{wagner2021constructions}, which used a simpler and thus faster neural network for the global step and which was therefore able to run thousands rather than dozens of loops.)

\section{Hard and easy problems}

The question of which problems are most amenable to investigation by machine learning methods is only beginning to be understood.  We are still at the stage where we simply have to try things and see what works. In the present project, we noticed that PatternBoost clearly works better for some problems than for others. In this section we'll focus on three problems which highlight this variation when considered together. The first problem is an example where our method, while it clearly outperforms the pure local search approach, does not compete well with the best human efforts; the second will be an example where PatternBoost is able to learn the best structures, dramatically improves over the pure local search approach, and almost matches the best results we were able to achieve with algorithms tailored to this specific problem.  And in the third problem, PatternBoost is able to arrive at a  solution which refutes a long-standing conjecture.

Later in Section~\ref{sec:other_problems} we will see several more examples, which we'll treat more briefly, where PatternBoost improved on the best previously known results and led to new discoveries in mathematics.

\subsection{The no-squares problem}\label{subsec:nosquares}

Let us consider the following problem, very similar in nature to the example problem of triangle-free graphs with many edges that we considered in Section~\ref{sec:trifree}.  At most how many edges can a graph on $n$ vertices have, if it doesn't contain a cycle of length 4? This, and similar problems seem to be notoriously difficult for machine learning methods to grasp~\cite{mehrabian2023finding}. While this problem is significantly harder both mathematically and computationally than the same problem with triangles, a lot is known about it~\cite{furedi2013history}, and in particular the exact optimal number of edges is known up to $n=40$~\cite[A006855]{oeis}. 

We will fix the number of vertices at $n=33$ and use the same simple local search as we used in Section~\ref{sec:trifree} to illustrate our method: given any graph as input, we first delete edges if necessary, greedily, until there are no more 4-cycles left, and then we keep adding edges randomly for as long as possible. To establish a baseline result, we ran 50 million local searches starting from the empty graph; the result are in Figure~\ref{fig:4cycle_baseline}.

\begin{figure}
\begin{minipage}{0.54\textwidth} 
\includegraphics[width=\textwidth]{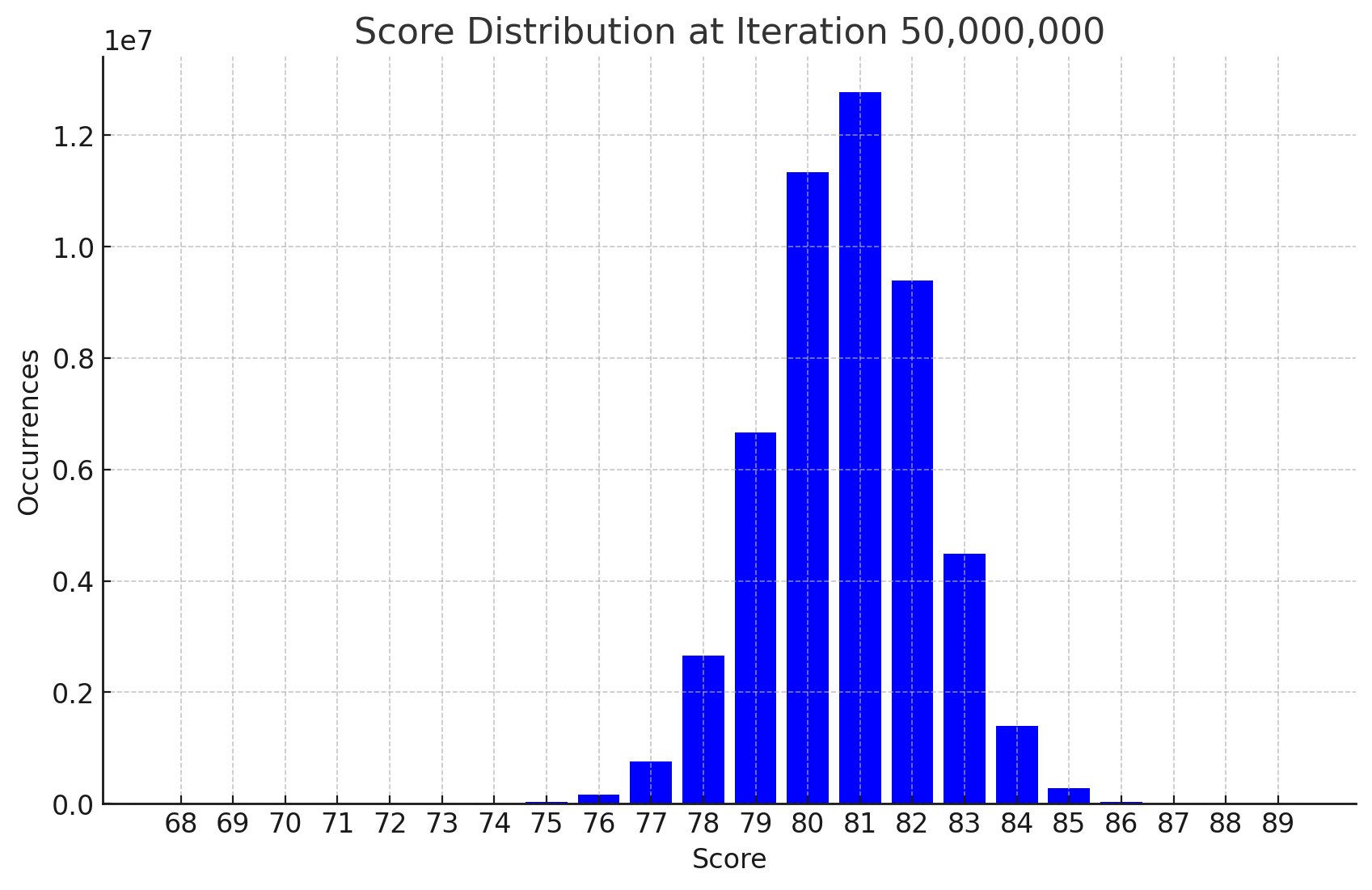}%\label{fig:teaser1}
    \small
    \caption{The results of adding random edges to empty graph for as long as possible, without creating any 4-cycles, 50M times in total.}
    \label{fig:4cycle_baseline}
\end{minipage}\hfill 
\begin{minipage}{0.44\textwidth}
\includegraphics[width=\textwidth]%{two_samples_gcd_100.pdf}%\label{fig:teaser2}
{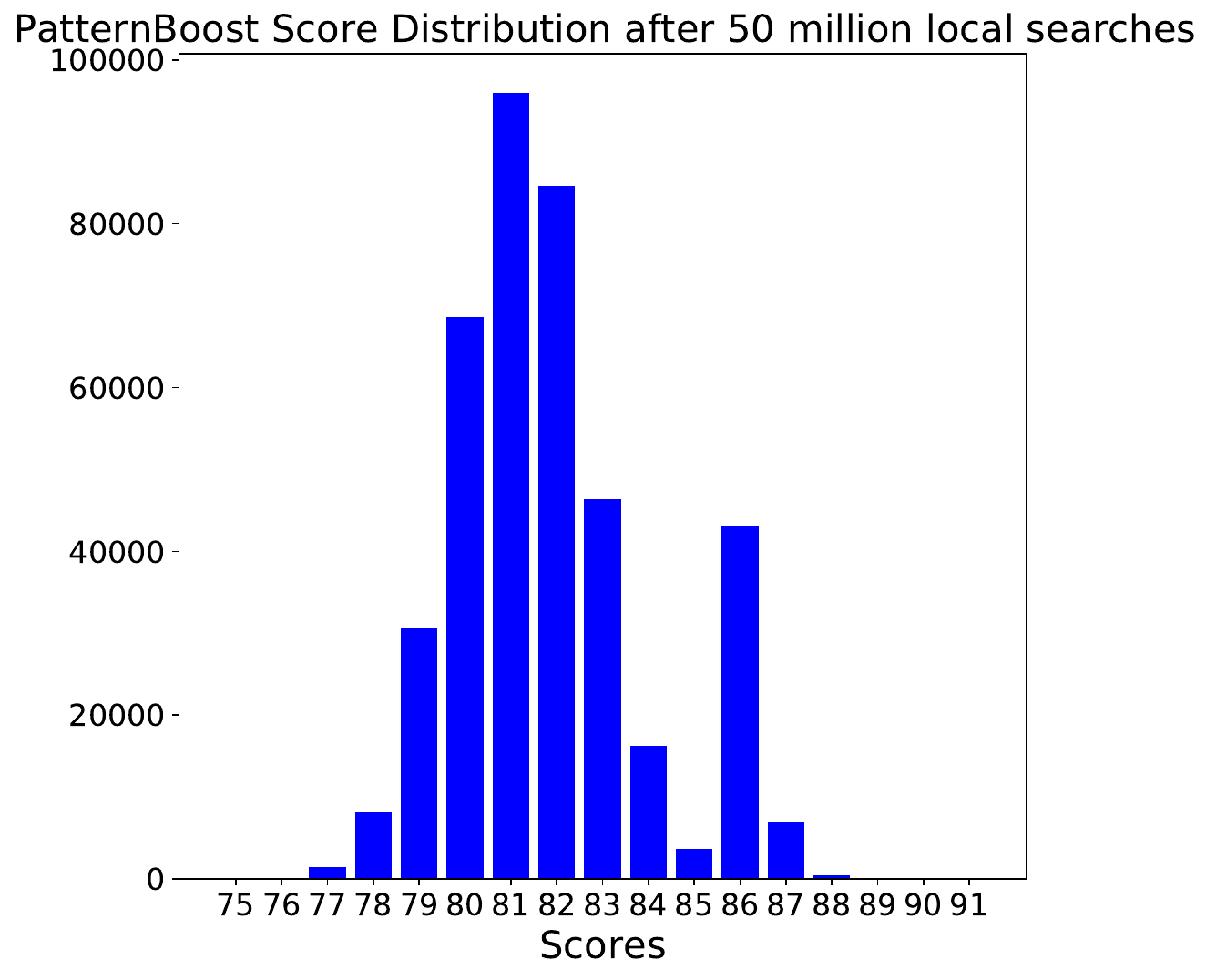}
    \small
    \caption{Results with PatternBoost, 50M local searches in total.}
    \label{fig:4cycle_boost1}
    \vspace{0.5cm}
\end{minipage}
\end{figure}

The lowest score was 68, the highest was 89, and the distribution of results peaks at 81. In particular, in 50M attempts, we obtained 5 graphs with score 89, 204 graphs with score 88, and 3425 graphs with score 87. The maximum possible achievable score for $n=33$ is 96~\cite[A006855]{oeis}.

\subsubsection{A first experiment}

Let us run PatternBoost on this problem, using the same local-global method that we used in Section~\ref{sec:trifree}. We generate a starting dataset, train a transformer on it, ask the transformer to output new graphs that are similar to the graphs in the training set, improve these with local search to obtain a new dataset, and repeat. Note that we do not need to start the training of the transformer from scratch in every loop. As long as we use the same tokenization in each step, we can simply continue training the trained model from the previous step.

Using the same parameters as above (a transformer with $2$ layers, $4$ heads and $16$ dimensions), we ran $4$ experiments, with different (random) model initializations, and a vocabulary of $100$ BPE tokens. After each generation, the transformer model created $500,000$ candidate solutions. These were tested for correctness (i.e. once decoded, the model prediction must be a valid adjacency matrix for a graph with $33$ vertices), then the correct solutions were improved using local search and added to the $50$k best solutions previously found (i.e. used to train the model). After $50$ million local searches ($158$ generations), the best of our $4$ models generated one graph with $91$ edges, found after $20.5$ million local searches, $21$ graphs with $89$ edges, $454$ with $88$ and $6,957$ with $87$, a significant improvement over our previous results. Figure~\ref{fig:4cycle_boost1} presents the scores of the candidates generated by the model after $158$ generations (about $50$ million local searches). The distribution is bimodal. The low mode ($81$) corresponds to the data generated by the transformer during this generation, its scores are slightly higher than those from the baseline distribution (Figure~\ref{fig:4cycle_baseline}). The high mode ($86$) corresponds to good candidates, accumulated from previous generations, which are kept as training examples.

\subsubsection{The advantage of larger models, and better tokenizers}

The previous experiments rely on very small transformers ($2$ layers, $16$ dimensions), and a simple tokenization scheme (BPE-encoding of the adjacency matrix). Better performances can be achieved with larger models. A transformer with $256$ dimensions, $6$ layers and $8$ heads finds a solution with $91$ edges after less than a million local searches, vs $8.6$ millions for our fastest model with $16$ dimensions, $2$ layers and $4$ heads. 

During the generation phase, our models tend to create a lot of invalid solutions. For a model-generated sequence to map to a valid candidate, it must decode into a sequence of $n(n-1)/2$ binary entries -- the adjacency matrix of a graph with $n$ vertices. When using BPE-tokenization, tokens correspond to a variable number of bits. In order to output a valid graph, the transformer must learn to add the length of its output tokens, which proves a hard task for Makemore, no matter the size of the model. In previous experiments, even after many generations, about a third of model predictions were incorrect. This makes learning slower, and less efficient.

A very simple change in graph representation allows for much better model prediction. Before tokenization, we add a delimiter at the end of every line in the adjacency matrix. Instead of representing the adjacency matrix 
$$\begin{matrix}
0 & 1 & 0 \\
  & 0 & 1 \\
  &   & 0 \\
\end{matrix}$$
as the sequence $\texttt{010010}$, we represent it as $\texttt{010,01,0,}$. The BPE-tokenizer is now run on sequences with a vocabulary of $3$ ($\{`0',`1',`,'\}$), and the sequences are longer (by $n$ entries), but this greatly helps Makemore to output valid sequence: only $5-10\%$ of model prediction are now invalid.

Using this new tokenizer, we trained $40$ models with $128$ dimensions, $4$ heads and $4$ and $6$ layers, for $n=33$. After $50$ million local searches, one model had found a graph with $93$ edges, and $2$ had found graphs with $92$ edges. A maximal graph with $96$ edges was discovered after $116.5$ million local searches, using a $6$-layer transformer, with dimension $128$ and $4$ attention heads.

Table~\ref{tab:nosquare_dist} compares the distributions of solutions found by local search and PatternBoost, after $50$ million local search steps. The best candidates found with local search only have $89$ edges. In contrast, PatternBoost finds graphs with $91$ edges and, its improved versions,  discovers over many good candidates: $2750$ graphs with $89$ edges or more, vs $22$ in the base version of PatternBoost, and $5$ wiht local search alone).

\begin{table}[h]

\centering
\small
\begin{tabular}{c|c|cc}
\toprule
 & &\multicolumn{2}{c}{PatternBoost}\\
 & Local search & $2$ layers $16$ dimensions & $4$ layers $128$ dimensions \\
Edges& & & improved tokenizer\\
\midrule
91 & 0 & 1 & 15 \\ 
90 & 0 & 0 & 206 \\
89 & 5 & 21 & 2,529 \\ 
88 & 204 & 454 & 26,842 \\
87 & 3,425 & 6,957 & 175,089 \\
\bottomrule
\end{tabular}
\caption{\small \textbf{No-squares problem, $n=33$}. Distribution of the number of edges of candidate solutions, after 50 million local searches.}
\label{tab:nosquare_dist}
\end{table}

Figure~\ref{fig:learningcurves} presents the learning curves of $20$ different models, with $6$ layers, $128$ dimensions and $4$ heads, only differing by the initialization of their weights, and the seed of their random number generator. For each generation, we compute the average number of edges of the $10,000$ best solutions found so far (left), and the best candidate found (right). For all models, the population of candidates improves  over time: new and better constructions are generated, as makemore is fine-tuned on better and better sets of candidates.

Yet, all models do not perform the same. In this case, two models learn significantly faster. It takes them about $50$ generations to discover $10,000$ graphs with more than $88$ edges (and no $4$-cycles), whereas the other models need $120$ to $200$ generations to achieve the same performance. Interestingly, these two models are the first to discover a graph with $92$ edges, and the only ones to discover graphs with $93$ edges or more. This confirms the importance of the successive fine-tuning of the transformer models for achieving the best results: if discovering the best solutions was, for the most part, an effect of ``lucky draws'' during local search, they would not be specific to the fastest learning models. 

These results also suggest a possible strategy for optimising PatternBoost, when a large amount of computing resources (GPU) can be used for a short period of time. Since the speed of learning seems to be a good advanced indicator of future model performance, we could run a large number of models, with different random initializations, for a few hours, and keep the fastest learners.

\begin{figure}[h!]
    \centering
    \includegraphics[width=0.8\linewidth]{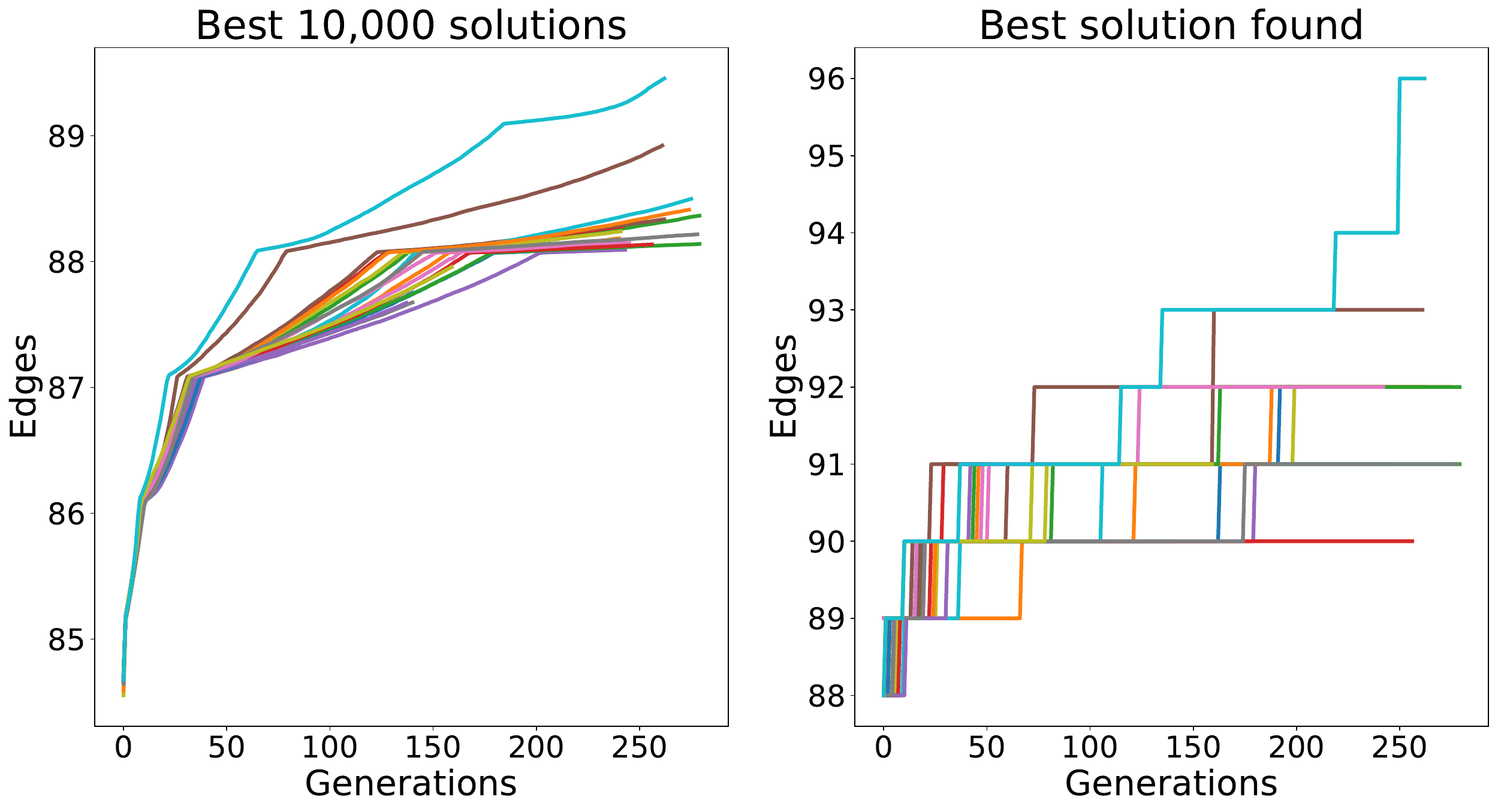}
    \caption{Evolution of scores over generations, for 20 different models. All models have 6 layers, and 128 dimensions, and only differ by their initialization. Left: average score of the 10,000 best solutions found so far. Right: best solution found so far. }
    \label{fig:learningcurves}
\end{figure}

\subsubsection{Scaling PatternBoost}

The no-squares problem becomes increasingly hard as $n$, the number of vertices, increases. To assess the capability of PatternBoost to scale to larger problems, we trained models with $128$ dimensions, $4$ or $6$ layers and the new tokenizer, for problems of increasing size: from $n=20$ to $n=55$. PatternBoost could find a best solution for all $n\leq 30$ in less than $80$ million local searches, and a best solution for $n=33$ was found after $116.5$ million local searches.
Table~\ref{tab:nosquare_scaling} summarizes our results.

\begin{table}[h!]

\centering
\small
\begin{tabular}{lc|cc|| lc|cc }
\toprule
$n$ & Best& Best found& \#Searches & $n$ & Best& Best found& \#Searches \\
\midrule
20 & 46 & \textbf{46}& 0.4 & 38 & 117 & 110 & 77.1\\ 
21 & 50 & \textbf{50}& 0.8 & 39 & 122 & 113 & 11.7\\
22 & 52 & \textbf{52}& 0.4 & 40 & 127 & 117 & 11.2\\
23 & 56 & \textbf{56}& 2.3 & 41 & $\geq$132 & 102 & 8.6\\
24 & 59 & \textbf{59}& 1.8 & 42 & $\geq$137 & 124 & 6.8 \\
25 & 63 & \textbf{63}& 3.1 & 43 & $\geq$142 & 128 & 15.6\\
26 & 67 & \textbf{67} & 8.4 & 44 & $\geq$148 & 132 & 12.1\\
27 & 71 &\textbf{71} & 21.0 & 45 & $\geq$154 & 136 & 11.1\\
28 & 76 & \textbf{76}& 23.8 & 46 & $\geq$157 & 140 & 10.4\\
29 & 80 & \textbf{80}& 41.7 & 47 & $\geq$163 & 144 & 12.1\\
30 & 85 & \textbf{85} & 78.3& 48 & $\geq$168 & 147 & 1.9\\ 
31 & 90 & 87 & 126.3 & 49 & $\geq$174 & 152 & 12.2\\ 
32 & 92 & 91 & 104.9 & 50 & ? & 156 & 15.2 \\
33 & 96 & \textbf{96} & 116.5 & 51 & ? & 161 & 14.7 \\
34 & 102 & 95 & 39.9 & 52 & ? & 165 & 14.1\\
35 & 106 & 100 & 121.3 & 53 & ? & 169 & 5.2\\
36 & 110 & 102 & 38.4 & 54 & ? & 173 & 9.0 \\
37 & 113 & 106 & 29.8 & 55 & ? & 177 & 8.2\\
\bottomrule
\end{tabular}
\caption{\small \textbf{No-squares problem}. Best solutions found by PatternBoost, and number of local searches (millions). Best of $6$ models for $n<39$ (80 for $n=33$), of $12$ models for $n\geq 39$.}
\label{tab:nosquare_scaling}
\end{table}

\subsection{The 312-avoiding permanent problem}

Previously, in Section~\ref{subsec:nosquares}, we saw a problem that was rather difficult for our method (or indeed any other machine learning methods we know of) to grasp. Let us now see an example where our method works much better and  significantly improves over the local search baseline. Even by using the simplest possible local search algorithm and no problem-specific insights built into it, PatternBoost  almost matches the performance of the best specialized, tailored algorithm we could come up with for this problem.

Let us say that a binary matrix $A$ \emph{contains the pattern 312} if there are three column indices $c_1 < c_2 < c_3$ and three row indices $r_1 < r_2 < r_3$, such that the entries at positions $(r_1,c_3)$, $(r_2,c_1)$, and $(r_3,c_2)$ are all ones. Recall that the \emph{permanent} of an $n\times n$ matrix $A$ is defined as
$$\text{per}(A) = \sum_{\sigma\in S_n}\prod_{i=1}^n a_{i,\sigma(i)},$$
where $S_n$ denotes the set of all permutations of $\{1,2,\ldots,n\}$.

The problem we will consider here is as follows:
\begin{question}[Brualdi--Cao~\cite{brualdi2020pattern}]
    How large can the permanent of a binary $n\times n$ matrix be, that does not contain the pattern 312?
\end{question}
For partial results on this problem, and the best answers for small values of $n$, see~\cite{wagner2021constructions}.  In particular, it is known that the maximum value is between the exponentials $2^{0.89n}$ and $((24)^{1/4})^n.$

We compared the performance of the following three algorithms on this problem, for $n=25$. 
\begin{enumerate}
    \item \textbf{Simple local search} The first one is an extremely simplistic random greedy search. We start from the empty matrix and keep adding ones in random positions for as long as possible, without creating any 312 patterns. This algorithm performed very badly: after running this search 30,000 times, always starting from the empty matrix and adding ones randomly for as long as we can, the highest permanent was still only 641,000. %[transformer_project/perm_new/perm_only_local].
    \item \textbf{PatternBoost} The second algorithm is the method of this paper, combined with the simple local search from the previous point. We train a transformer on the best constructions we have encountered so far, then ask it to generate new constructions. From each construction generated by the transformer we first greedily delete ones if needed, to get rid of any potential 312s, and then add ones in random positions for as long as we can, without creating any new 312s, and then repeat.
    \item \textbf{Handcrafted specialized search method} The third algorithm is a handcrafted search not using ML, that we created specifically for this problem, using lots of human insights, ad hoc intuitions about the best constructions, and any other tricks we could think of. 
\end{enumerate}

Let us call the third method above the ``human method'' (as it is a standard algorithm crafted by humans, and typically the technique currently used on such problems).
Note that PatternBoost is completely generic, it has (almost) nothing built into it that is specific to this problem, whereas the human method is a very specialized algorithm that would not be applicable to any other problem. We had a friendly competition on whether we can outperform the generic ML method with our handcrafted problem-specific algorithms, and after two weeks and a tight race the ``human'' method eked out a slight win: it found a matrix with permanent 5,200,384, whereas PatternBoost was only able to achieve 5,101,230. This shows that handcrafted specialized algorithms can still outperform generic ML methods. Remarkably the top 10 constructions found by these two methods were completely disjoint sets, and the human method has only found 8 matrices that were better than the best construction found by our ML method!

\begin{figure}[h]
    \centering
    \begin{tikzpicture}[scale=0.75]
\draw[step=0.3cm,gray,thin] (0,0) grid (7.5,7.5);
\fill[black!60!white] (0.015,7.484999999999999) rectangle (0.285,7.215);
\fill[black!60!white] (0.315,7.484999999999999) rectangle (0.585,7.215);
\fill[black!60!white] (0.6149999999999999,7.484999999999999) rectangle (0.885,7.215);
\fill[black!60!white] (0.315,7.185) rectangle (0.585,6.915);
\fill[black!60!white] (0.6149999999999999,7.185) rectangle (0.885,6.915);
\fill[black!60!white] (0.9149999999999999,7.185) rectangle (1.185,6.915);
\fill[black!60!white] (0.315,6.885) rectangle (0.585,6.615);
\fill[black!60!white] (0.6149999999999999,6.885) rectangle (0.885,6.615);
\fill[black!60!white] (0.9149999999999999,6.885) rectangle (1.185,6.615);
\fill[black!60!white] (1.2149999999999999,6.885) rectangle (1.485,6.615);
\fill[black!60!white] (0.015,6.585) rectangle (0.285,6.315);
\fill[black!60!white] (0.315,6.585) rectangle (0.585,6.315);
\fill[black!60!white] (0.9149999999999999,6.585) rectangle (1.185,6.315);
\fill[black!60!white] (1.2149999999999999,6.585) rectangle (1.485,6.315);
\fill[black!60!white] (0.015,6.284999999999999) rectangle (0.285,6.015);
\fill[black!60!white] (1.2149999999999999,6.284999999999999) rectangle (1.485,6.015);
\fill[black!60!white] (1.515,6.284999999999999) rectangle (1.785,6.015);
\fill[black!60!white] (1.815,6.284999999999999) rectangle (2.085,6.015);
\fill[black!60!white] (1.515,5.984999999999999) rectangle (1.785,5.715);
\fill[black!60!white] (1.815,5.984999999999999) rectangle (2.085,5.715);
\fill[black!60!white] (2.1149999999999998,5.984999999999999) rectangle (2.385,5.715);
\fill[black!60!white] (2.415,5.984999999999999) rectangle (2.6849999999999996,5.715);
\fill[black!60!white] (2.1149999999999998,5.685) rectangle (2.385,5.415);
\fill[black!60!white] (2.415,5.685) rectangle (2.6849999999999996,5.415);
\fill[black!60!white] (2.7150000000000003,5.685) rectangle (2.985,5.415);
\fill[black!60!white] (2.415,5.385) rectangle (2.6849999999999996,5.115);
\fill[black!60!white] (2.7150000000000003,5.385) rectangle (2.985,5.115);
\fill[black!60!white] (3.015,5.385) rectangle (3.2849999999999997,5.115);
\fill[black!60!white] (2.1149999999999998,5.085) rectangle (2.385,4.815);
\fill[black!60!white] (2.415,5.085) rectangle (2.6849999999999996,4.815);
\fill[black!60!white] (2.7150000000000003,5.085) rectangle (2.985,4.815);
\fill[black!60!white] (3.015,5.085) rectangle (3.2849999999999997,4.815);
\fill[black!60!white] (1.815,4.784999999999999) rectangle (2.085,4.515);
\fill[black!60!white] (2.1149999999999998,4.784999999999999) rectangle (2.385,4.515);
\fill[black!60!white] (3.015,4.784999999999999) rectangle (3.2849999999999997,4.515);
\fill[black!60!white] (3.315,4.784999999999999) rectangle (3.5849999999999995,4.515);
\fill[black!60!white] (1.515,4.484999999999999) rectangle (1.785,4.215);
\fill[black!60!white] (1.815,4.484999999999999) rectangle (2.085,4.215);
\fill[black!60!white] (3.015,4.484999999999999) rectangle (3.2849999999999997,4.215);
\fill[black!60!white] (3.315,4.484999999999999) rectangle (3.5849999999999995,4.215);
\fill[black!60!white] (0.015,4.185) rectangle (0.285,3.915);
\fill[black!60!white] (1.2149999999999999,4.185) rectangle (1.485,3.915);
\fill[black!60!white] (1.515,4.185) rectangle (1.785,3.915);
\fill[black!60!white] (3.315,4.185) rectangle (3.5849999999999995,3.915);
\fill[black!60!white] (3.615,4.185) rectangle (3.885,3.915);
\fill[black!60!white] (3.915,4.185) rectangle (4.185,3.915);
\fill[black!60!white] (3.615,3.885) rectangle (3.885,3.615);
\fill[black!60!white] (3.915,3.885) rectangle (4.185,3.615);
\fill[black!60!white] (4.215,3.885) rectangle (4.484999999999999,3.615);
\fill[black!60!white] (4.515,3.885) rectangle (4.784999999999999,3.615);
\fill[black!60!white] (4.215,3.5849999999999995) rectangle (4.484999999999999,3.315);
\fill[black!60!white] (4.515,3.5849999999999995) rectangle (4.784999999999999,3.315);
\fill[black!60!white] (4.815,3.5849999999999995) rectangle (5.085,3.315);
\fill[black!60!white] (4.215,3.2849999999999997) rectangle (4.484999999999999,3.015);
\fill[black!60!white] (4.515,3.2849999999999997) rectangle (4.784999999999999,3.015);
\fill[black!60!white] (4.815,3.2849999999999997) rectangle (5.085,3.015);
\fill[black!60!white] (5.115,3.2849999999999997) rectangle (5.385,3.015);
\fill[black!60!white] (3.915,2.985) rectangle (4.185,2.7150000000000003);
\fill[black!60!white] (4.215,2.985) rectangle (4.484999999999999,2.7150000000000003);
\fill[black!60!white] (4.815,2.985) rectangle (5.085,2.7150000000000003);
\fill[black!60!white] (5.115,2.985) rectangle (5.385,2.7150000000000003);
\fill[black!60!white] (3.915,2.6849999999999996) rectangle (4.185,2.415);
\fill[black!60!white] (5.115,2.6849999999999996) rectangle (5.385,2.415);
\fill[black!60!white] (5.415,2.6849999999999996) rectangle (5.685,2.415);
\fill[black!60!white] (3.615,2.385) rectangle (3.885,2.1149999999999998);
\fill[black!60!white] (3.915,2.385) rectangle (4.185,2.1149999999999998);
\fill[black!60!white] (5.115,2.385) rectangle (5.385,2.1149999999999998);
\fill[black!60!white] (5.415,2.385) rectangle (5.685,2.1149999999999998);
\fill[black!60!white] (3.615,2.085) rectangle (3.885,1.815);
\fill[black!60!white] (5.415,2.085) rectangle (5.685,1.815);
\fill[black!60!white] (5.715,2.085) rectangle (5.984999999999999,1.815);
\fill[black!60!white] (6.015,2.085) rectangle (6.284999999999999,1.815);
\fill[black!60!white] (5.715,1.785) rectangle (5.984999999999999,1.515);
\fill[black!60!white] (6.015,1.785) rectangle (6.284999999999999,1.515);
\fill[black!60!white] (6.315,1.785) rectangle (6.585,1.515);
\fill[black!60!white] (6.615,1.785) rectangle (6.885,1.515);
\fill[black!60!white] (6.315,1.485) rectangle (6.585,1.2149999999999999);
\fill[black!60!white] (6.615,1.485) rectangle (6.885,1.2149999999999999);
\fill[black!60!white] (6.915,1.485) rectangle (7.185,1.2149999999999999);
\fill[black!60!white] (6.615,1.185) rectangle (6.885,0.9149999999999999);
\fill[black!60!white] (6.915,1.185) rectangle (7.185,0.9149999999999999);
\fill[black!60!white] (7.215,1.185) rectangle (7.484999999999999,0.9149999999999999);
\fill[black!60!white] (6.315,0.885) rectangle (6.585,0.6149999999999999);
\fill[black!60!white] (6.615,0.885) rectangle (6.885,0.6149999999999999);
\fill[black!60!white] (6.915,0.885) rectangle (7.185,0.6149999999999999);
\fill[black!60!white] (7.215,0.885) rectangle (7.484999999999999,0.6149999999999999);
\fill[black!60!white] (5.715,0.585) rectangle (5.984999999999999,0.315);
\fill[black!60!white] (6.015,0.585) rectangle (6.284999999999999,0.315);
\fill[black!60!white] (6.315,0.585) rectangle (6.585,0.315);
\fill[black!60!white] (7.215,0.585) rectangle (7.484999999999999,0.315);
\fill[black!60!white] (0.015,0.285) rectangle (0.285,0.015000000000000013);
\fill[black!60!white] (3.315,0.285) rectangle (3.5849999999999995,0.015000000000000013);
\fill[black!60!white] (3.615,0.285) rectangle (3.885,0.015000000000000013);
\fill[black!60!white] (5.415,0.285) rectangle (5.685,0.015000000000000013);
\fill[black!60!white] (5.715,0.285) rectangle (5.984999999999999,0.015000000000000013);
\fill[black!60!white] (7.215,0.285) rectangle (7.484999999999999,0.015000000000000013);
\node[] at (3.75,-0.7) {  $\text{per}\left(A_{25}\right)=5101230$};\end{tikzpicture}
    \begin{tikzpicture}[scale=0.75]
\draw[step=0.3cm,gray,thin] (0,0) grid (7.5,7.5);
\fill[black!60!white] (0.015,7.484999999999999) rectangle (0.285,7.215);
\fill[black!60!white] (0.315,7.484999999999999) rectangle (0.585,7.215);
\fill[black!60!white] (0.6149999999999999,7.484999999999999) rectangle (0.885,7.215);
\fill[black!60!white] (0.9149999999999999,7.484999999999999) rectangle (1.185,7.215);
\fill[black!60!white] (0.6149999999999999,7.185) rectangle (0.885,6.915);
\fill[black!60!white] (0.9149999999999999,7.185) rectangle (1.185,6.915);
\fill[black!60!white] (1.2149999999999999,7.185) rectangle (1.485,6.915);
\fill[black!60!white] (0.6149999999999999,6.885) rectangle (0.885,6.615);
\fill[black!60!white] (0.9149999999999999,6.885) rectangle (1.185,6.615);
\fill[black!60!white] (1.2149999999999999,6.885) rectangle (1.485,6.615);
\fill[black!60!white] (1.515,6.885) rectangle (1.785,6.615);
\fill[black!60!white] (0.315,6.585) rectangle (0.585,6.315);
\fill[black!60!white] (0.6149999999999999,6.585) rectangle (0.885,6.315);
\fill[black!60!white] (1.2149999999999999,6.585) rectangle (1.485,6.315);
\fill[black!60!white] (1.515,6.585) rectangle (1.785,6.315);
\fill[black!60!white] (0.315,6.284999999999999) rectangle (0.585,6.015);
\fill[black!60!white] (1.515,6.284999999999999) rectangle (1.785,6.015);
\fill[black!60!white] (1.815,6.284999999999999) rectangle (2.085,6.015);
\fill[black!60!white] (0.015,5.984999999999999) rectangle (0.285,5.715);
\fill[black!60!white] (0.315,5.984999999999999) rectangle (0.585,5.715);
\fill[black!60!white] (1.515,5.984999999999999) rectangle (1.785,5.715);
\fill[black!60!white] (1.815,5.984999999999999) rectangle (2.085,5.715);
\fill[black!60!white] (0.015,5.685) rectangle (0.285,5.415);
\fill[black!60!white] (1.815,5.685) rectangle (2.085,5.415);
\fill[black!60!white] (2.1149999999999998,5.685) rectangle (2.385,5.415);
\fill[black!60!white] (2.415,5.685) rectangle (2.6849999999999996,5.415);
\fill[black!60!white] (2.1149999999999998,5.385) rectangle (2.385,5.115);
\fill[black!60!white] (2.415,5.385) rectangle (2.6849999999999996,5.115);
\fill[black!60!white] (2.7150000000000003,5.385) rectangle (2.985,5.115);
\fill[black!60!white] (2.1149999999999998,5.085) rectangle (2.385,4.815);
\fill[black!60!white] (2.415,5.085) rectangle (2.6849999999999996,4.815);
\fill[black!60!white] (2.7150000000000003,5.085) rectangle (2.985,4.815);
\fill[black!60!white] (3.015,5.085) rectangle (3.2849999999999997,4.815);
\fill[black!60!white] (0.015,4.784999999999999) rectangle (0.285,4.515);
\fill[black!60!white] (1.815,4.784999999999999) rectangle (2.085,4.515);
\fill[black!60!white] (2.1149999999999998,4.784999999999999) rectangle (2.385,4.515);
\fill[black!60!white] (2.7150000000000003,4.784999999999999) rectangle (2.985,4.515);
\fill[black!60!white] (3.015,4.784999999999999) rectangle (3.2849999999999997,4.515);
\fill[black!60!white] (0.015,4.484999999999999) rectangle (0.285,4.215);
\fill[black!60!white] (3.015,4.484999999999999) rectangle (3.2849999999999997,4.215);
\fill[black!60!white] (3.315,4.484999999999999) rectangle (3.5849999999999995,4.215);
\fill[black!60!white] (3.615,4.484999999999999) rectangle (3.885,4.215);
\fill[black!60!white] (3.315,4.185) rectangle (3.5849999999999995,3.915);
\fill[black!60!white] (3.615,4.185) rectangle (3.885,3.915);
\fill[black!60!white] (3.915,4.185) rectangle (4.185,3.915);
\fill[black!60!white] (3.015,3.885) rectangle (3.2849999999999997,3.615);
\fill[black!60!white] (3.315,3.885) rectangle (3.5849999999999995,3.615);
\fill[black!60!white] (3.615,3.885) rectangle (3.885,3.615);
\fill[black!60!white] (3.915,3.885) rectangle (4.185,3.615);
\fill[black!60!white] (3.015,3.5849999999999995) rectangle (3.2849999999999997,3.315);
\fill[black!60!white] (3.915,3.5849999999999995) rectangle (4.185,3.315);
\fill[black!60!white] (4.215,3.5849999999999995) rectangle (4.484999999999999,3.315);
\fill[black!60!white] (4.515,3.5849999999999995) rectangle (4.784999999999999,3.315);
\fill[black!60!white] (4.215,3.2849999999999997) rectangle (4.484999999999999,3.015);
\fill[black!60!white] (4.515,3.2849999999999997) rectangle (4.784999999999999,3.015);
\fill[black!60!white] (4.815,3.2849999999999997) rectangle (5.085,3.015);
\fill[black!60!white] (4.215,2.985) rectangle (4.484999999999999,2.7150000000000003);
\fill[black!60!white] (4.515,2.985) rectangle (4.784999999999999,2.7150000000000003);
\fill[black!60!white] (4.815,2.985) rectangle (5.085,2.7150000000000003);
\fill[black!60!white] (5.115,2.985) rectangle (5.385,2.7150000000000003);
\fill[black!60!white] (4.815,2.6849999999999996) rectangle (5.085,2.415);
\fill[black!60!white] (5.115,2.6849999999999996) rectangle (5.385,2.415);
\fill[black!60!white] (5.415,2.6849999999999996) rectangle (5.685,2.415);
\fill[black!60!white] (5.715,2.6849999999999996) rectangle (5.984999999999999,2.415);
\fill[black!60!white] (5.415,2.385) rectangle (5.685,2.1149999999999998);
\fill[black!60!white] (5.715,2.385) rectangle (5.984999999999999,2.1149999999999998);
\fill[black!60!white] (6.015,2.385) rectangle (6.284999999999999,2.1149999999999998);
\fill[black!60!white] (5.415,2.085) rectangle (5.685,1.815);
\fill[black!60!white] (5.715,2.085) rectangle (5.984999999999999,1.815);
\fill[black!60!white] (6.015,2.085) rectangle (6.284999999999999,1.815);
\fill[black!60!white] (6.315,2.085) rectangle (6.585,1.815);
\fill[black!60!white] (5.115,1.785) rectangle (5.385,1.515);
\fill[black!60!white] (5.415,1.785) rectangle (5.685,1.515);
\fill[black!60!white] (6.015,1.785) rectangle (6.284999999999999,1.515);
\fill[black!60!white] (6.315,1.785) rectangle (6.585,1.515);
\fill[black!60!white] (5.115,1.485) rectangle (5.385,1.2149999999999999);
\fill[black!60!white] (6.315,1.485) rectangle (6.585,1.2149999999999999);
\fill[black!60!white] (6.615,1.485) rectangle (6.885,1.2149999999999999);
\fill[black!60!white] (6.915,1.485) rectangle (7.185,1.2149999999999999);
\fill[black!60!white] (6.615,1.185) rectangle (6.885,0.9149999999999999);
\fill[black!60!white] (6.915,1.185) rectangle (7.185,0.9149999999999999);
\fill[black!60!white] (7.215,1.185) rectangle (7.484999999999999,0.9149999999999999);
\fill[black!60!white] (6.615,0.885) rectangle (6.885,0.6149999999999999);
\fill[black!60!white] (6.915,0.885) rectangle (7.185,0.6149999999999999);
\fill[black!60!white] (7.215,0.885) rectangle (7.484999999999999,0.6149999999999999);
\fill[black!60!white] (4.815,0.585) rectangle (5.085,0.315);
\fill[black!60!white] (5.115,0.585) rectangle (5.385,0.315);
\fill[black!60!white] (6.315,0.585) rectangle (6.585,0.315);
\fill[black!60!white] (6.615,0.585) rectangle (6.885,0.315);
\fill[black!60!white] (7.215,0.585) rectangle (7.484999999999999,0.315);
\fill[black!60!white] (0.015,0.285) rectangle (0.285,0.015000000000000013);
\fill[black!60!white] (3.015,0.285) rectangle (3.2849999999999997,0.015000000000000013);
\fill[black!60!white] (3.915,0.285) rectangle (4.185,0.015000000000000013);
\fill[black!60!white] (4.215,0.285) rectangle (4.484999999999999,0.015000000000000013);
\fill[black!60!white] (4.815,0.285) rectangle (5.085,0.015000000000000013);
\fill[black!60!white] (7.215,0.285) rectangle (7.484999999999999,0.015000000000000013);
\node[] at (3.75,-0.7) {  $\text{per}\left(A_{25}'\right)=5200384$};\end{tikzpicture}
%5.2: 1111000000000000000000000.0011100000000000000000000.0011110000000000000000000.0110110000000000000000000.0100011000000000000000000.1100011000000000000000000.1000001110000000000000000.0000000111000000000000000.0000000111100000000000000.1000001101100000000000000.1000000000111000000000000.0000000000011100000000000.0000000000111100000000000.0000000000100111000000000.0000000000000011100000000.0000000000000011110000000.0000000000000000111100000.0000000000000000001110000.0000000000000000001111000.0000000000000000011011000.0000000000000000010001110.0000000000000000000000111.0000000000000000000000111.0000000000000000110001101.1000000000100110100000001.

% 5.1: 1110000000000000000000000011100000000000000000000001111000000000000000000001101100000000000000000000100011100000000000000000000000111100000000000000000000000111000000000000000000000001110000000000000000000001111000000000000000000001100110000000000000000001100011000000000000010001100000111000000000000000000000001111000000000000000000000001110000000000000000000000111100000000000000000000110110000000000000000000010001100000000000000000011000110000000000000000001000001110000000000000000000000011110000000000000000000000011100000000000000000000000111000000000000000000000111100000000000000000001110011000000000011000001100001

    \caption{The best construction found by general method PatternBoost (left) versus the best construction found by the best problem-specific standard algorithm we could create (right).}
    \label{fig:312}
\end{figure}
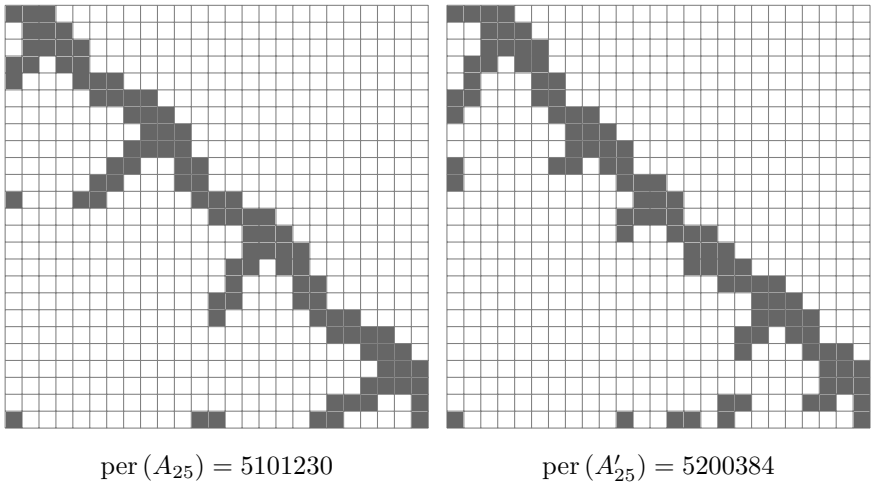

\subsubsection{Combining methods}

One obvious way to improve the performance of PatternBoost is to use a smarter local search algorithm than the simple one we used, thereby letting the transformer see ``more of the world''. By adding the 1000 best constructions found by our human method above to the training set of PatternBoost, the performance of PatternBoost immediately improves. Afterwards we have found 41 new constructions, unseen by either method before, that had permanents in the $5.1\cdot 10^6 - 5.2\cdot 10^6$ range, and so would have made it into the all time top 10 constructions seen with either method. Remarkably, none of these 41 new constructions beat the high score in Figure~\ref{fig:312}! While several of them had a permanent of around $5.19\cdot 10^6$, it appears that the human method got rather lucky by finding such a good construction while missing dozens of slightly suboptimal ones. But as is often the case in research, a bit of luck goes a long way.

 \subsection{Spanning subgraphs of the $d$-cube with diameter $d$}

There is no prevailing conjecture about the precise maximal number of edges in a graph on $n$ vertices with no $4$-cycle, or about the largest permanent of a $(312)$-avoiding binary matrix.  But there are other extremal problems where such conjectures have been formulated, and in such cases PatternBoost can be deployed to hunt for counterexamples.  In this section we describe a successful application of this strategy.

We will focus on an old problem that appeared in the works of Erd\H{o}s-Hamburger-Pippert-Weakley~\cite{erdos1996hypercube}, Graham--Harary~\cite{graham1992changing}, and Bouabdallah--Delorme--Djelloul~\cite{bouabdallah1995edge} who studied spanning subgraphs of the $d$-cube which have the same diameter $d$ as the cube itself. They asked the following natural question:

\begin{question}\label{que:erdos_hypercube_diam}
     What is the maximum number of edges one can delete from the $d$-dimensional hypercube, without increasing its diameter?
\end{question}

They observed that if we fix two opposite vertices $v$ and $v'$, and construct a subgraph $G$ by including, for every vertex $u\not\in\{v,v'\}$, an edge leading to a vertex that is closer to $v$ in the $d$-cube and an edge leading to a vertex that is closer to $v'$ in the $d$-cube, then the resulting subgraph is spanning and has diameter $d$. Such subgraphs have at least $2^d + \binom{d}{\lfloor d/2 \rfloor} - 2$ edges, and equality can be attained in many ways. See Figure~\ref{fig:spanning_hypercube_40} for an illustration of such a graph. They asked whether a better construction with fewer edges was possible, and Graham~\cite{graham1992changing} conjectured that this construction was in fact optimal. The best known lower bound is of the form $2^d + c\cdot\frac{2^d}{d}$~\cite{nenadov2019bounded}.

\begin{figure}[h!]
    \centering
    \includegraphics[width=0.6\linewidth]{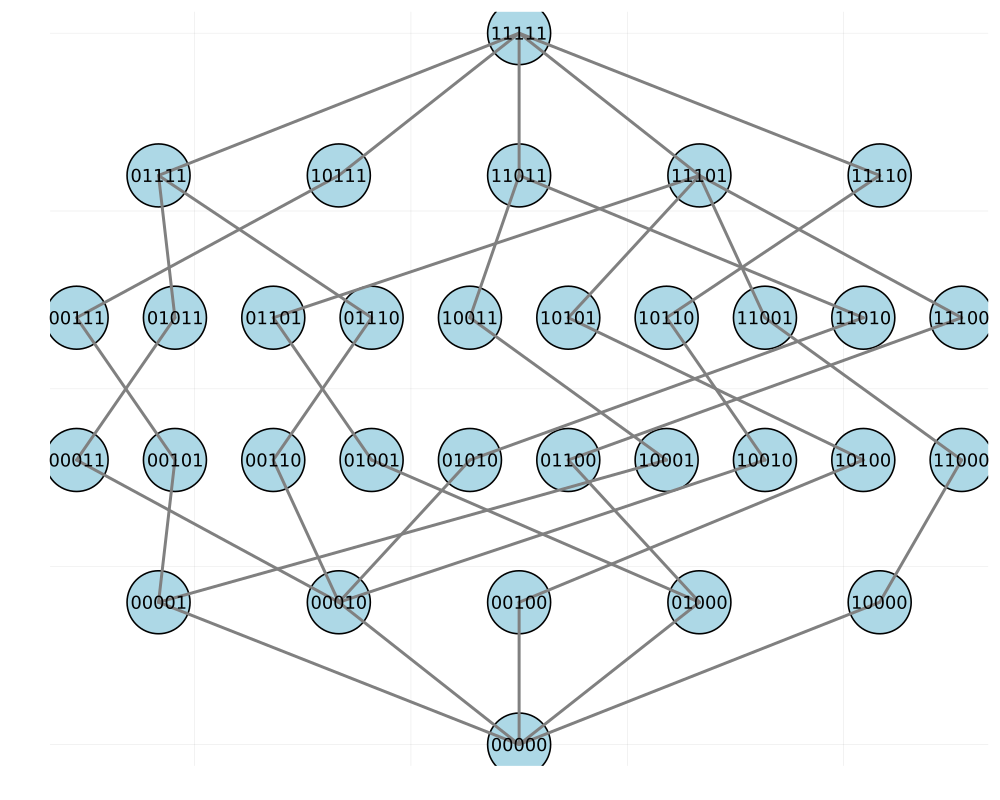}
    \caption{A subgraph of the 5-cube with diameter 5, with $2^4 + \binom{5}{2}-2 = 40$ edges. Note that from every vertex there is an edge going down and an edge going up, i.e.~there are no \emph{blocking} vertices.}
    \label{fig:spanning_hypercube_40}
\end{figure}

There is a natural way to set up this conjecture for PatternBoost. The score of a spanning, diameter $d$ subgraph can be the number of edges in it (which we try to minimize). For local search, the simplest algorithm one can do is, given a subgraph $G$, to add random edges to $G$ until it becomes spanning with diameter $d$, and then remove random edges for as long as possible while keeping the diameter at $d$. 

We ran this simple setup for $d=5$ and  $6$. For $d=5$ it seems that the above construction is optimal, but for $d=6$ we were able to find a graph with 81 edges (as opposed to the $2^6+\binom{6}{3}-2=82$ edges in the construction above), see Figure~\ref{fig:spanning_hypercube_81}. The complete list of edges can be found in Appendix~\ref{app:erdos_hypercube}. This disproves the above conjecture, and marks the first progress on this problem in 30 years. 

\begin{figure}[h!]
    \centering
    \includegraphics[width=\linewidth]{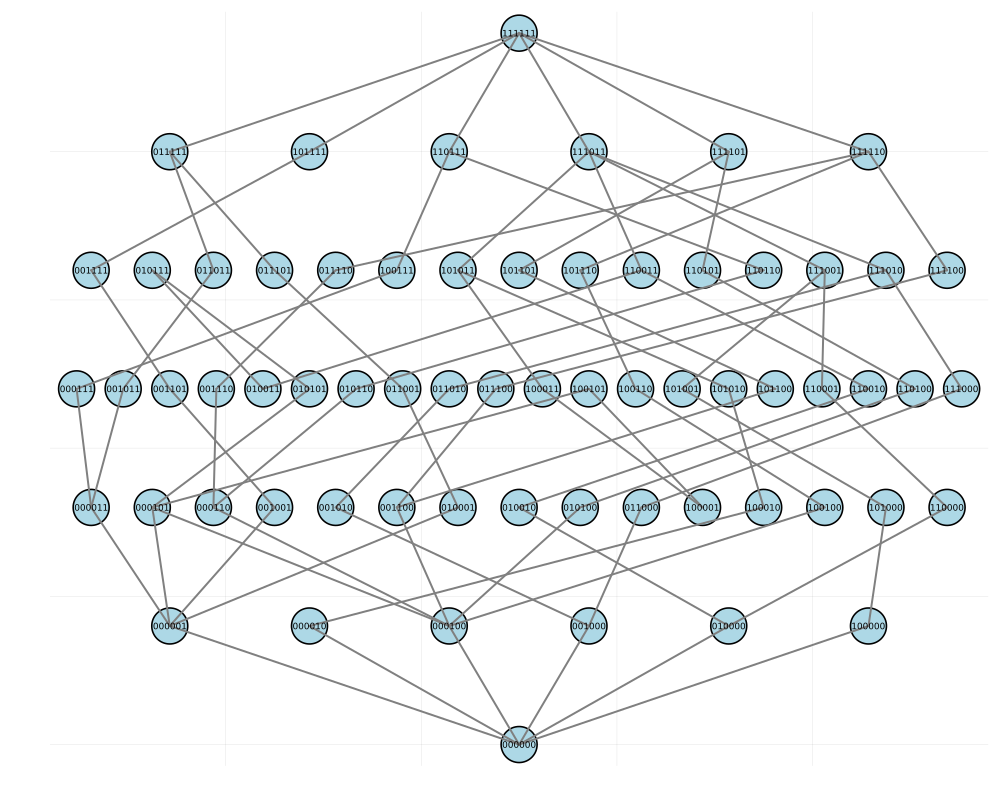}
    \caption{A spanning subgraph of the 6-cube with diameter 6 and 81 edges, which is a counterexample to the above conjecture. Note the presence of several \emph{blocking} vertices from which one can only go up or one can only go down.}
    \label{fig:spanning_hypercube_81}
\end{figure}

It would be interesting to see whether the lower or the upper bound is closer to the truth for large values of $d$.
 
\section{Other problems}\label{sec:other_problems}

\subsection{No isosceles triangles}\label{subsec:isosceles}

At most how many points can we choose in the $n\times n$ square grid, so that no three of them form the vertices of a (possibly flat) isosceles triangle? In other words, what is the value of 
$$f(n) := \max_{S\subset [n]^2}\{|S|: a,b,c\in S \text{ distinct} \implies d(a,b) \neq d(b,c)\},$$where $d(a,b)$ denotes Euclidean distance? This beautiful question was asked independently by Wu~\cite{wu2016counting}, Ellenberg--Jain~\cite{ellenberg2019convergence}, and possibly Erd\H{o}s~\cite{solymosierdos}. 

Determining the rate of growth of $f(n)$ seems like a hard problem. For upper bounds, the trivial $f(n)\leq n^2$ bound can be improved to $f(n)\leq e^{-c\log^{1/9}n}\cdot n^2$ by observing that a three term arithmetic progression in any horizontal line corresponds to an isosceles triangle, and using bounds on the sizes of such sets~\cite{bloom2023improvement}. Remarkably, even proving an upper bound of the form $f(n)\leq n^{1.99}$ is an open problem! 

A lower bound of $f(n)\geq cn$ can likely be deduced from results analyzing the random independent set process~\cite{bennett2016note}, but we could not find a short way to do so. A proof of a slightly weaker lower bound, obtained via the alteration method from probabilistic combinatorics, can be found in Appendix~\ref{app:iso_proof}.

Since we cannot determine the right order of magnitude of $f(n)$, it makes sense to approach the problem computationally for small values of $n$. What do the best constructions look like? For $n$ up to $\approx 32$, SAT solvers can find the best constructions and prove their optimality, see Figures~\ref{fig:isosceles_small1} and~\ref{fig:isosceles_small3}. Interestingly, for some values of $n$ there are multiple optimal constructions that look very different.

\begin{figure}[h]
  \begin{subfigure}{0.08\textwidth}
    \includegraphics[width=\linewidth]{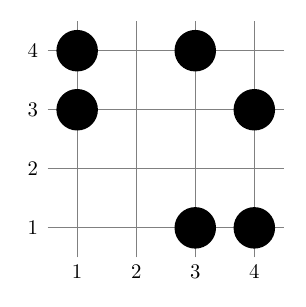}
    \caption{$f(4)=6$}
  \end{subfigure}\hfill
  \begin{subfigure}{0.10\textwidth}
    \includegraphics[width=\linewidth]{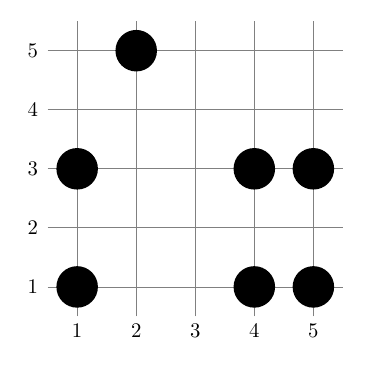}
    \caption{$f(5)=7$}
  \end{subfigure}
  \begin{subfigure}{0.12\textwidth}
    \includegraphics[width=\linewidth]{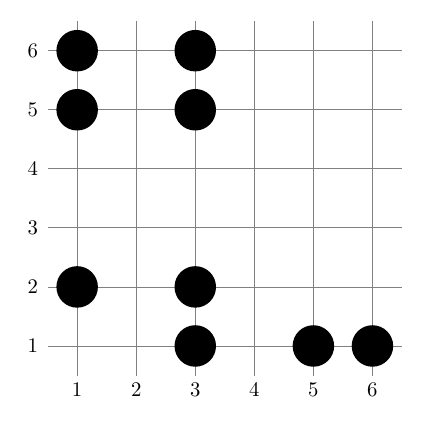}
    \caption{$f(6)=9$}
  \end{subfigure}\hfill
  \begin{subfigure}{0.14\textwidth}
    \includegraphics[width=\linewidth]{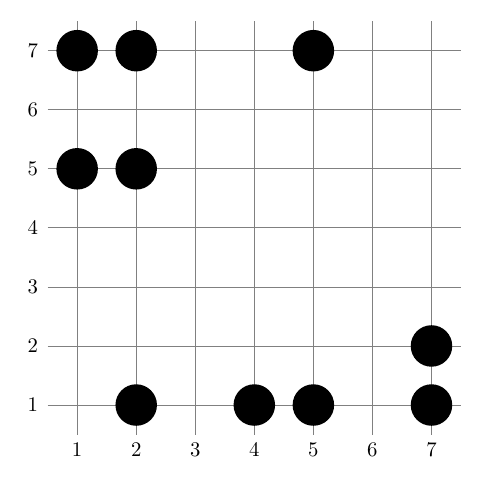}
    \caption{$f(7)=10$}
  \end{subfigure}
   \begin{subfigure}{0.16\textwidth}
    \includegraphics[width=\linewidth]{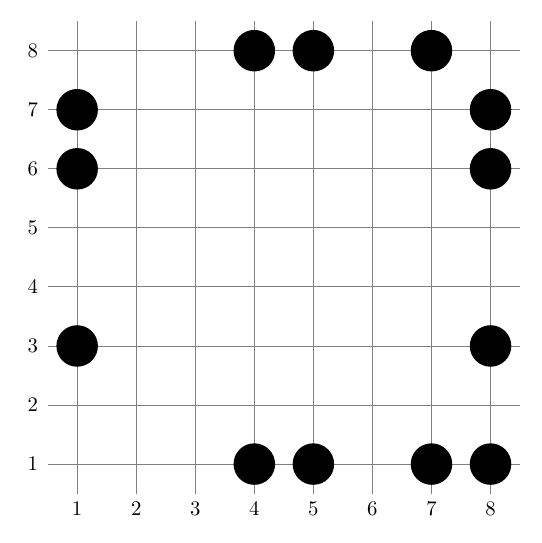}
    \caption{$f(8)=13$}
  \end{subfigure}\hfill
  \begin{subfigure}{0.18\textwidth}
    \includegraphics[width=\linewidth]{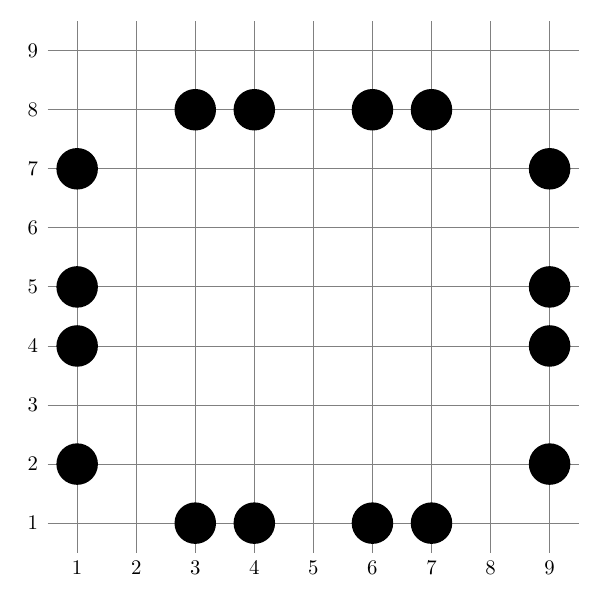}
    \caption{$f(9)=16$}
  \end{subfigure}
  \begin{subfigure}{0.20\textwidth}
    \includegraphics[width=\linewidth]{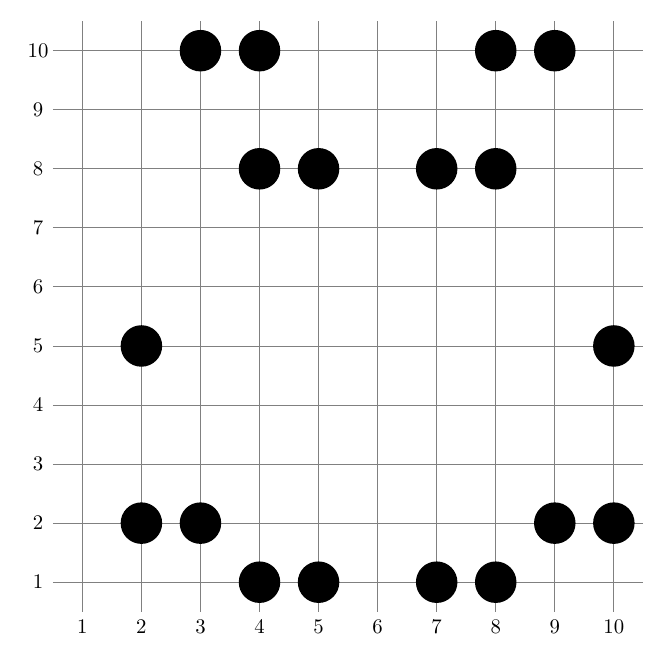}
    \caption{$f(10)=18$}
  \end{subfigure}\hfill
  \caption{The best constructions for $n=4$ to 10}
  \label{fig:isosceles_small1}
\end{figure}

\begin{figure}[h!]
  \begin{subfigure}{0.14\textwidth}
    \includegraphics[width=\linewidth]{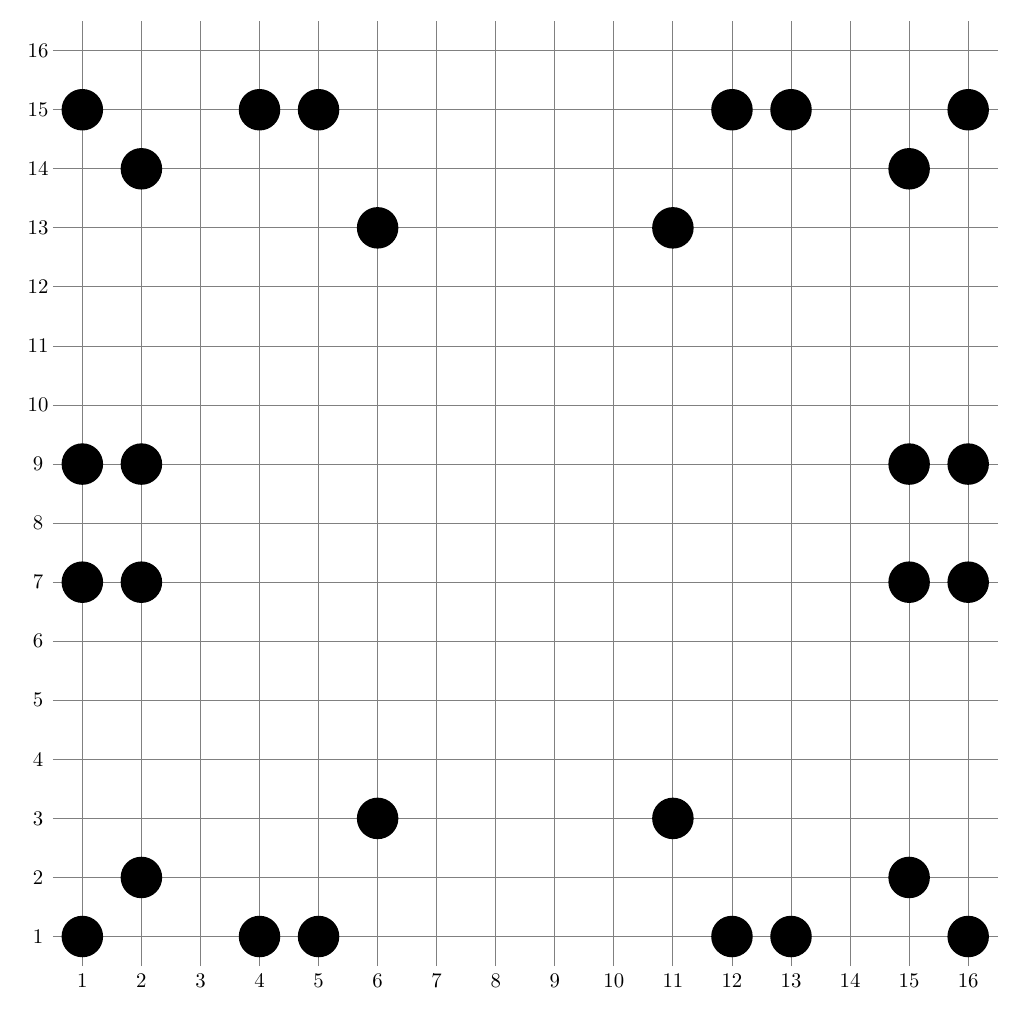}
    \caption{$f(16)= 28$}
  \end{subfigure}\hfill
  \begin{subfigure}{0.14\textwidth}
    \includegraphics[width=\linewidth]{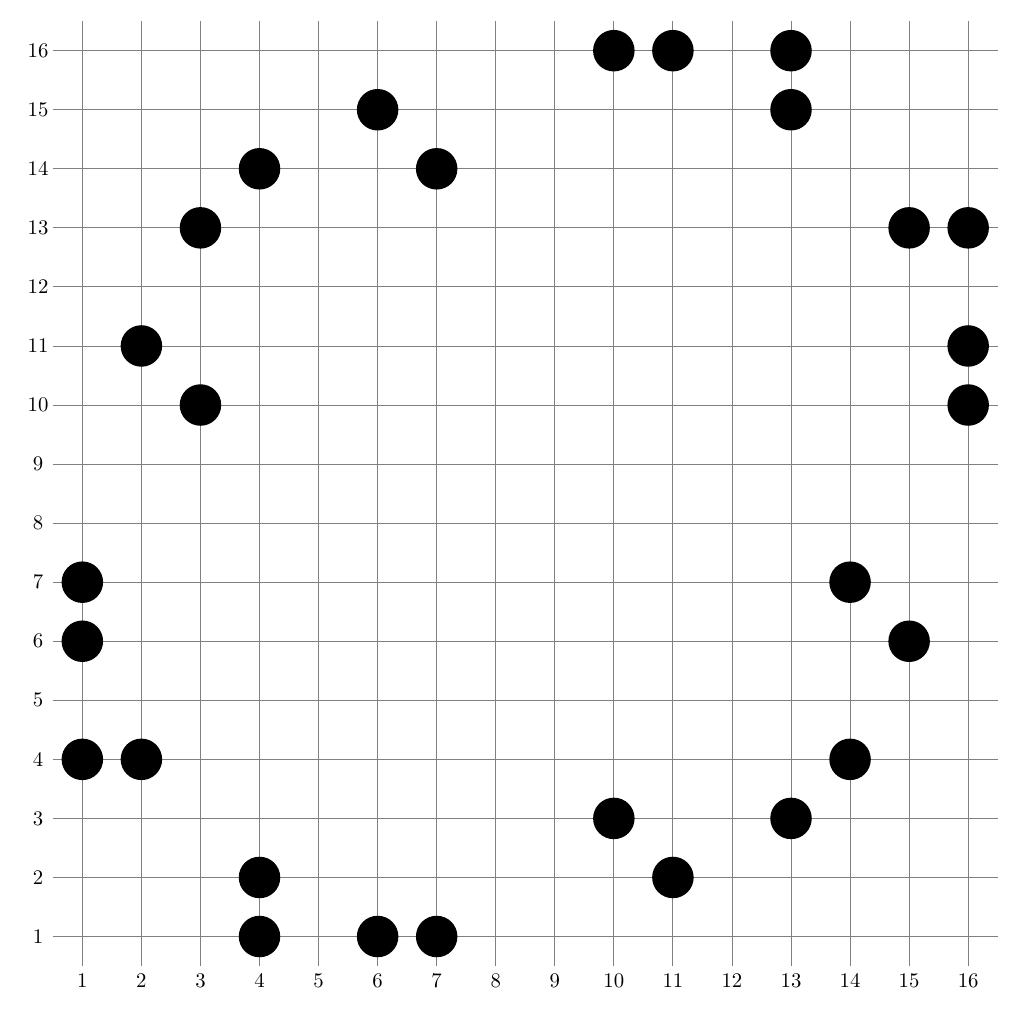}
    \caption{$f(16)= 28$}
  \end{subfigure}
  \begin{subfigure}{0.22\textwidth}
    \includegraphics[width=\linewidth]{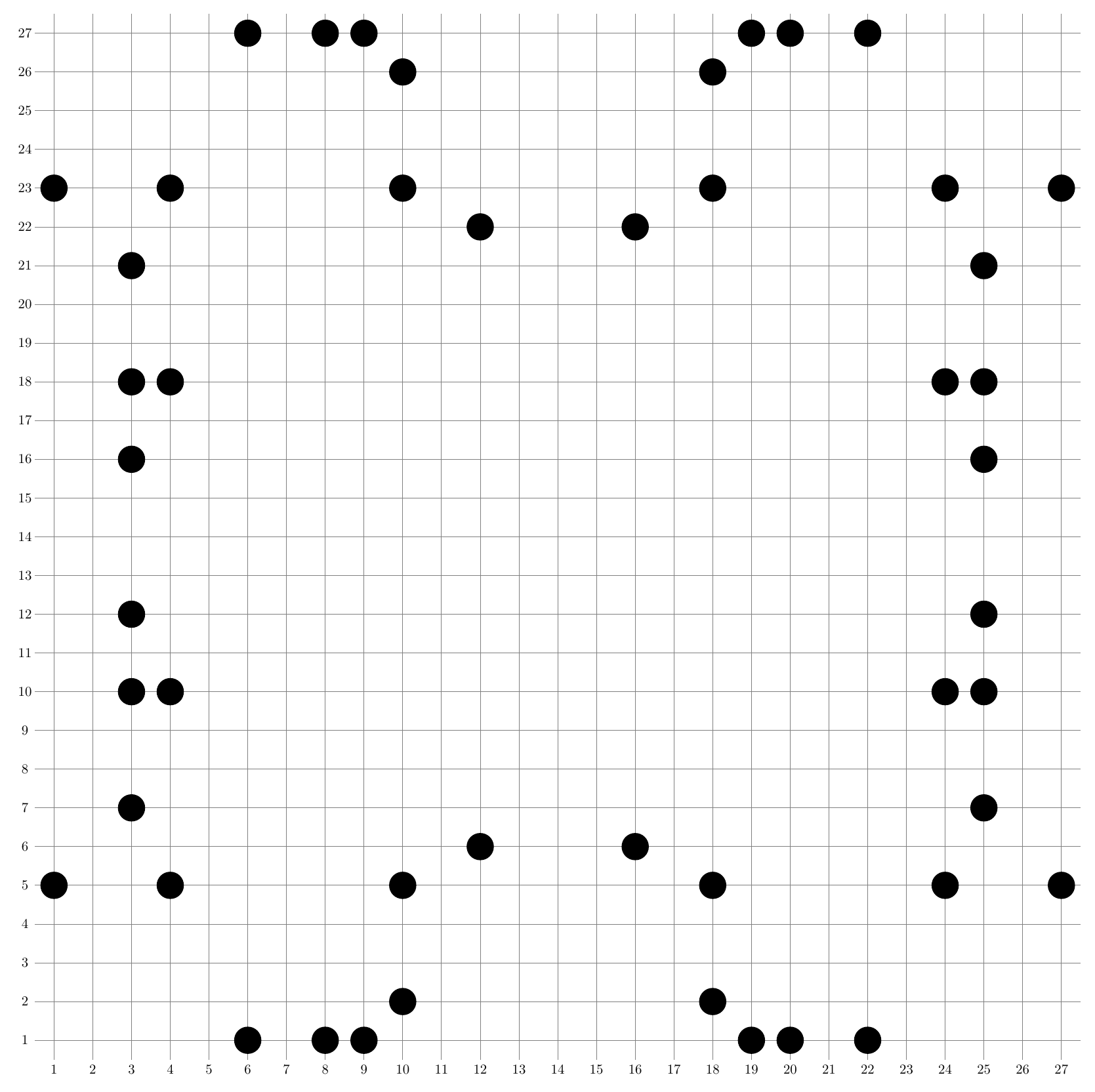}
    \caption{$f(27)= 48$}
  \end{subfigure}\hfill
  \begin{subfigure}{0.22\textwidth}
    \includegraphics[width=\linewidth]{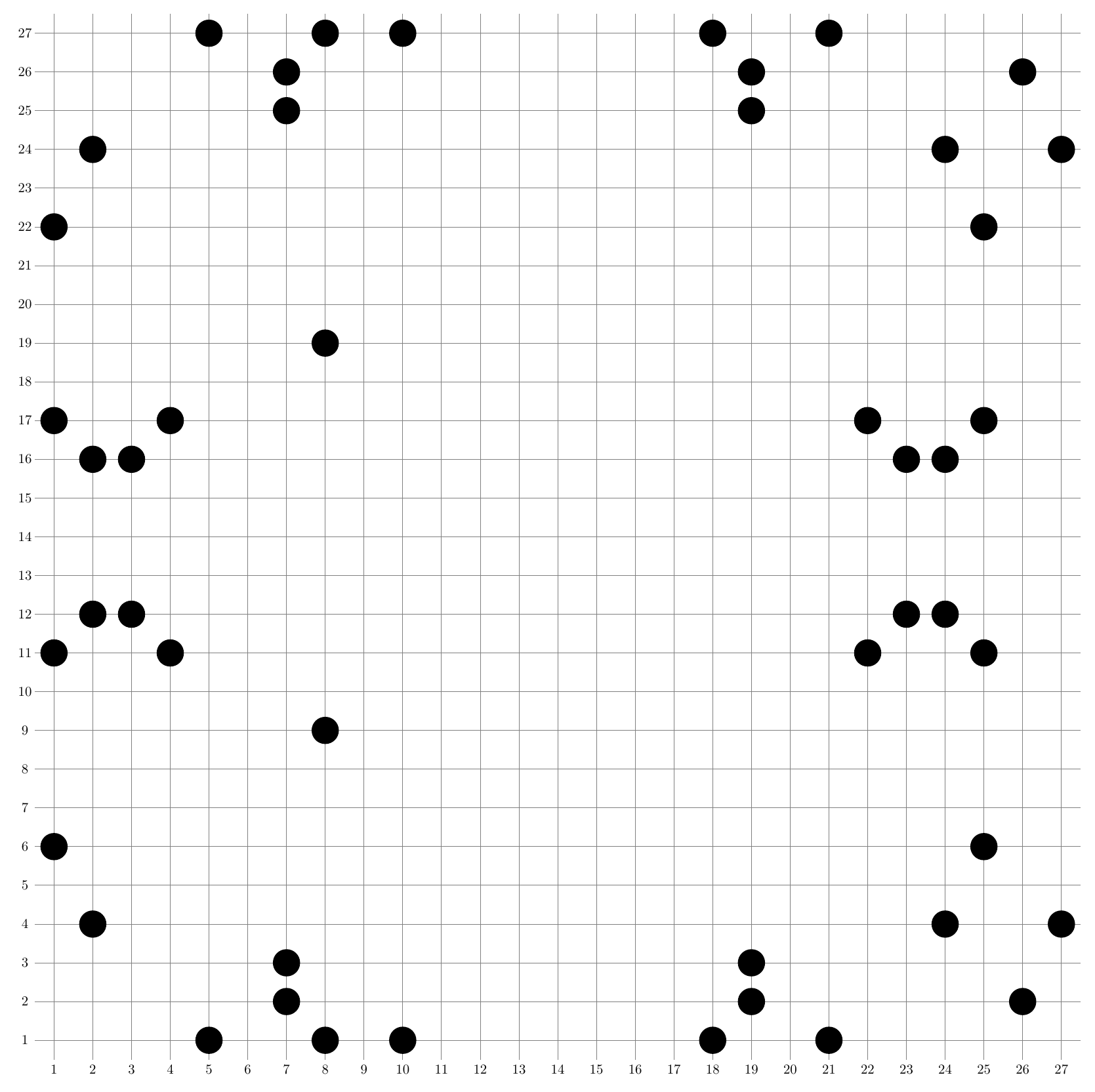}
    \caption{$f(27)= 48$}
  \end{subfigure}\hfill
  \begin{subfigure}{0.25\textwidth}
    \includegraphics[width=\linewidth]{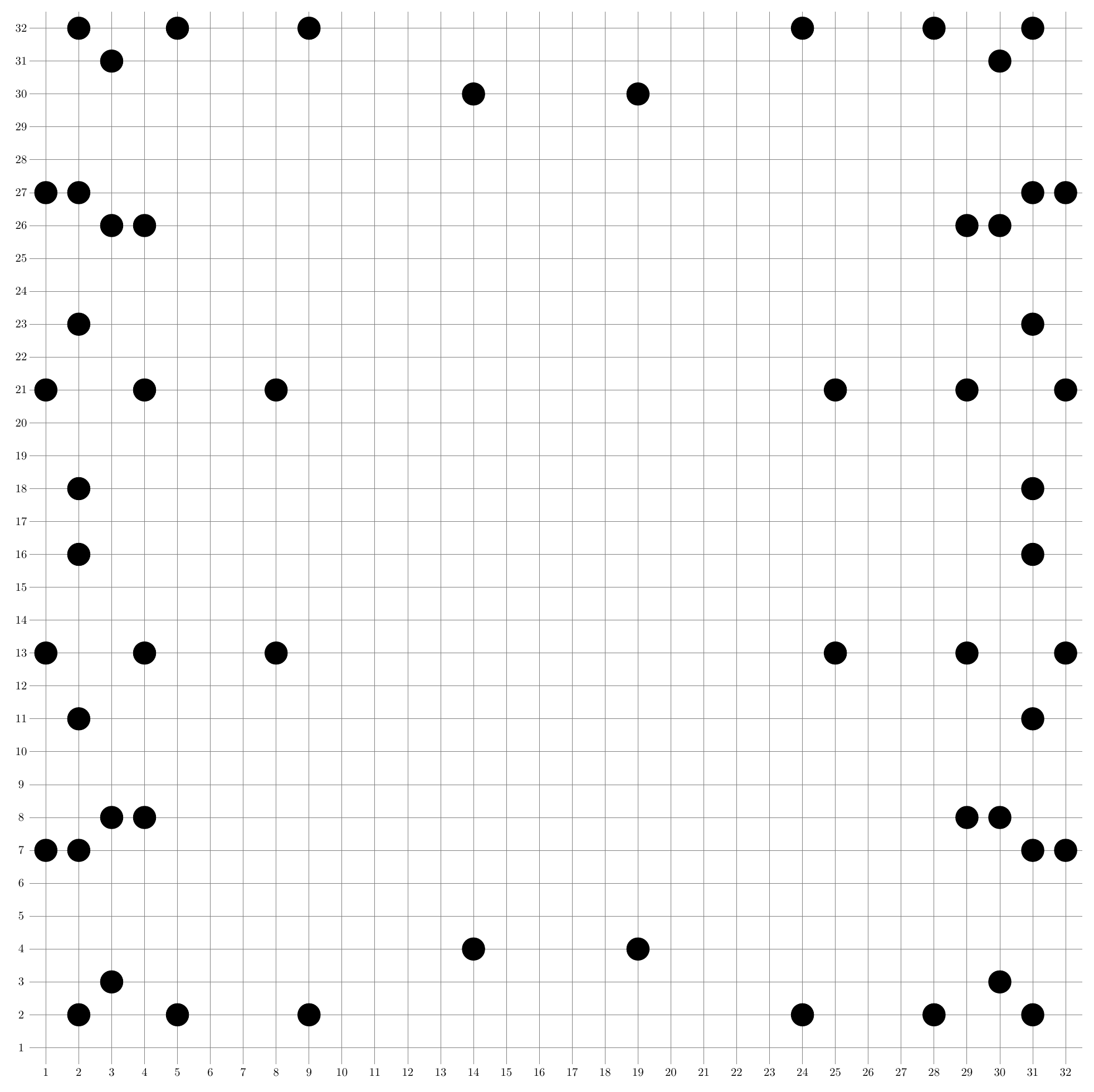}
    \caption{$f(32)= 56$}
    \end{subfigure}\hfill
  \caption{The best constructions for $n=16$, 27, and 32. For some $n$, there are many optimal solutions that look very different from each other.}
  \label{fig:isosceles_small3}

\end{figure}

When plotting the computed function values up to $n=32$, it appears that $f(n)$ is linear in $n$, and perhaps $f(n)$ will be approximately $\frac{16}{9}n$ for larger $n$ as well.
\begin{figure}[h!]
    \centering
    \includegraphics[width=0.35\linewidth]{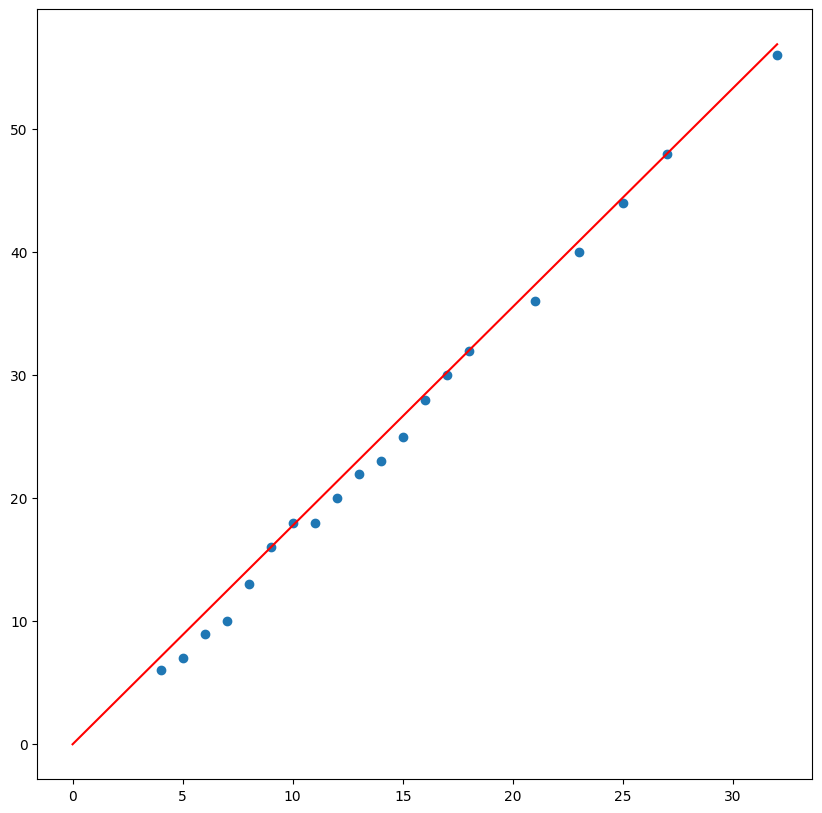}
    \caption{Plotting the computed values, $f$ seems to be linear}
    \label{fig:isosceles_plot}
\end{figure}
Based on these pictures of the optimal constructions for small~$n$, it was clear that the best solutions exhibit some kind of a pattern, but it is absolutely not clear to us what this pattern is and how we could describe or generalize it. These are the situations when PatternBoost can be useful -- when the pattern of the best constructions seems to be complicated for humans, there is a hope that machine learning methods can understand it a bit better than we can.

Let us now show how PatternBoost can be used to find good constructions for larger values of $n$, where standard methods don't work well anymore. For $n=64$, we first tried various problem-specific search methods running in total for about three months, to find the best construction we could. The result of this search was a solution with 108 points in $[64]^2$ without an isosceles triangle, see Figure~\ref{fig:108_110}. While doing our searches we always saved the best constructions we could find, and by the time we found a handful of different 108 constructions, we have built up a  database of about a million different maximal constructions with 104-108 points in them. Training a transformer on this dataset to find the common pattern in them, asking it to generate new constructions, and running a very simple local search on the resulting grids led to the discovery of a construction with 110 points within a few days. (Note that it wasn't too surprising that we skipped 109: most constructions we found had some symmetries in them, partially due to how our search algorithms worked.) This construction can be seen in Figure~\ref{fig:108_110}. By continuing this method we have found dozens more constructions with 110 points but never anything better -- even though the graph on Figure~\ref{fig:isosceles_plot} would suggest that 112 is also possible.

\begin{figure}[h!]
    \centering
    \includegraphics[width=0.45\linewidth]{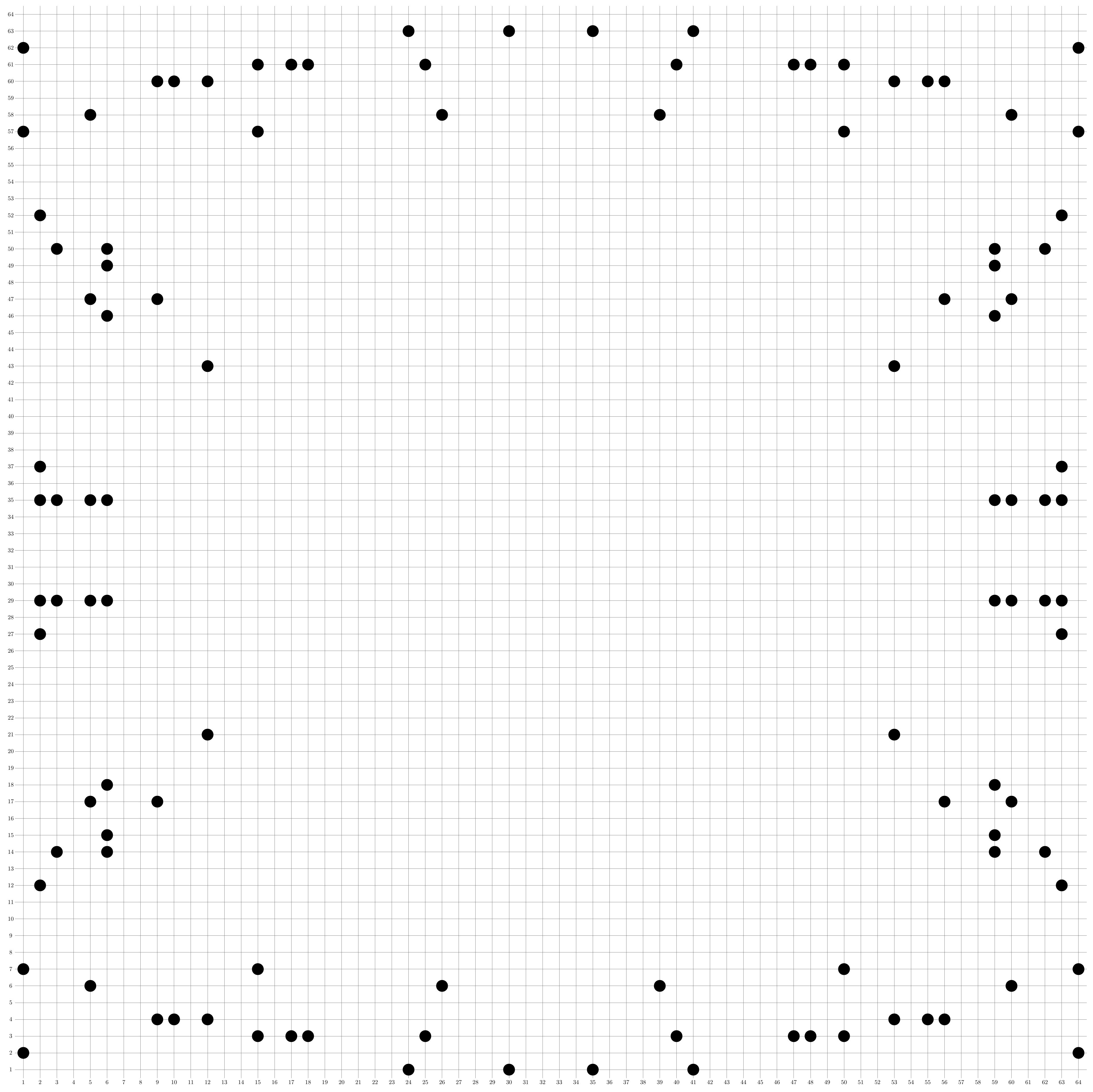}
    \includegraphics[width=0.45\linewidth]{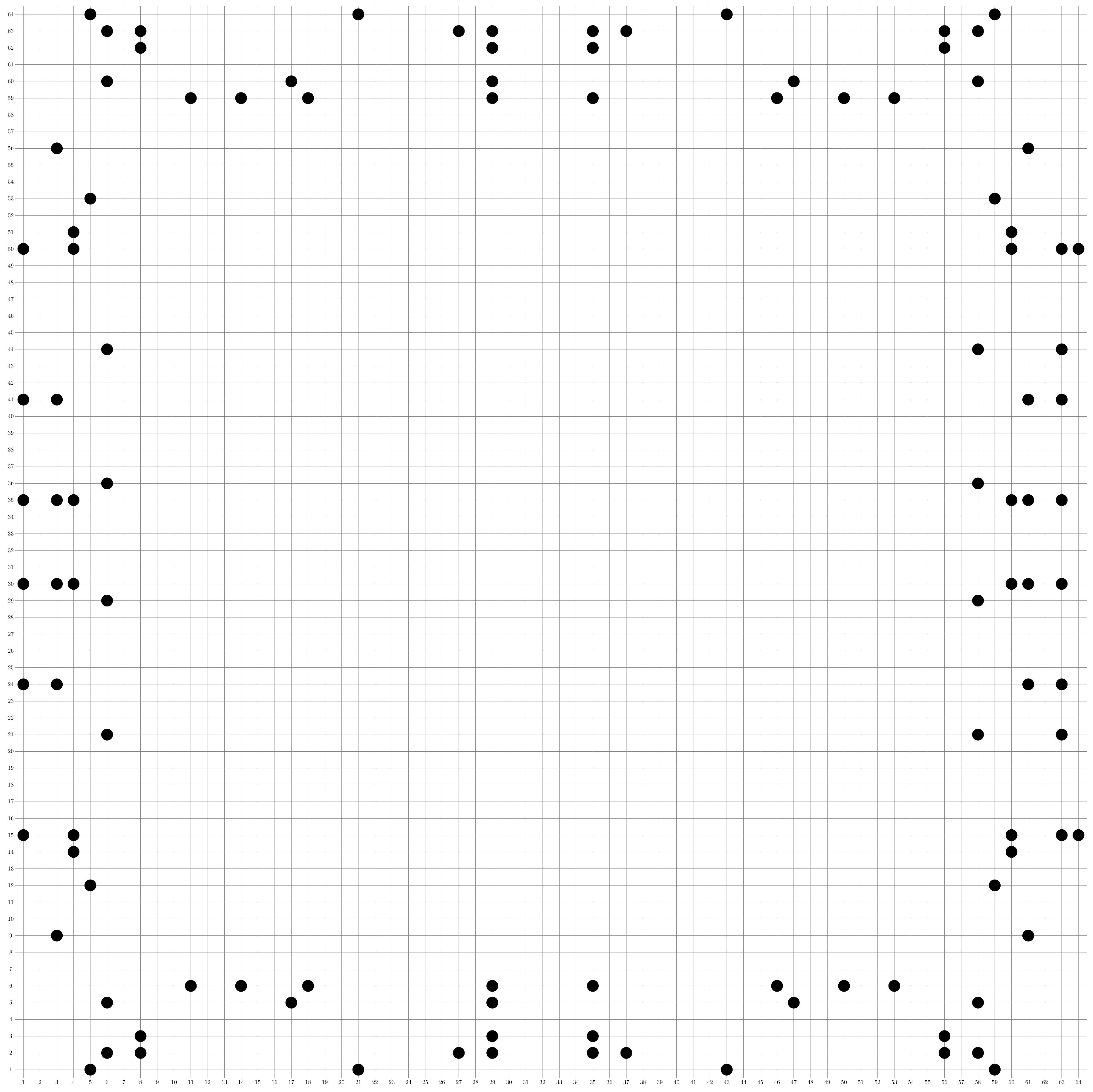}
    \caption{Left: a 108 point construction in $[64]^2$ without isosceles triangles, found by standard methods. Right: a 110 point construction found by adding only one transformer loop on top of the standard methods}
    \label{fig:108_110}
\end{figure}

Next, we repeated our searches for $n=100$. This size is much too large for us to even hope that we can find the optimal constructions, which likely have around 176 points in them. The best score our standard search techniques found was 154, and adding the transformer loop on top of this raised it to 160, see Figure~\ref{fig:isosceles_100}.

\begin{figure}[h!]
    \centering
    \includegraphics[width=0.45\linewidth]{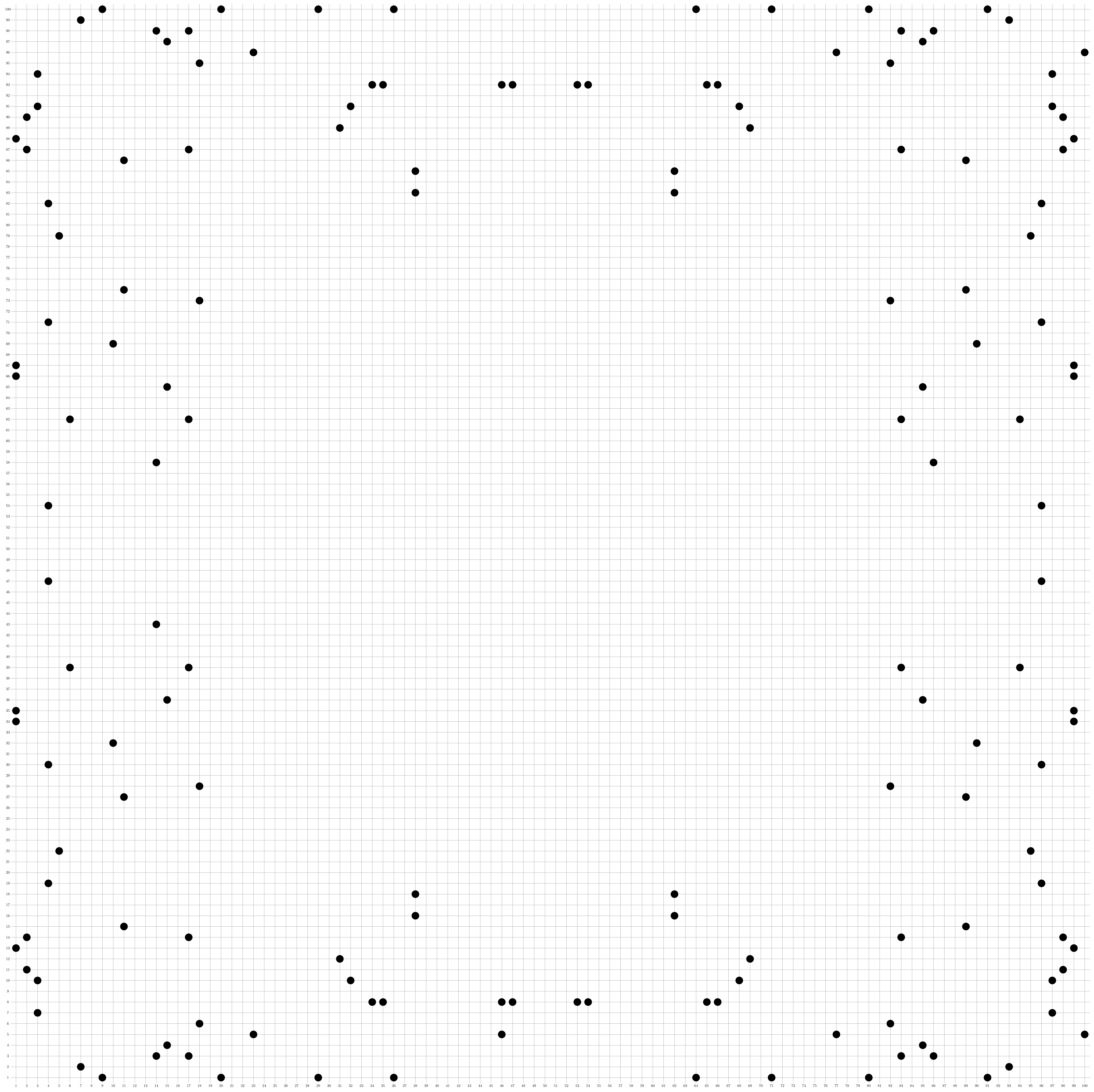}
    \includegraphics[width=0.45\linewidth]{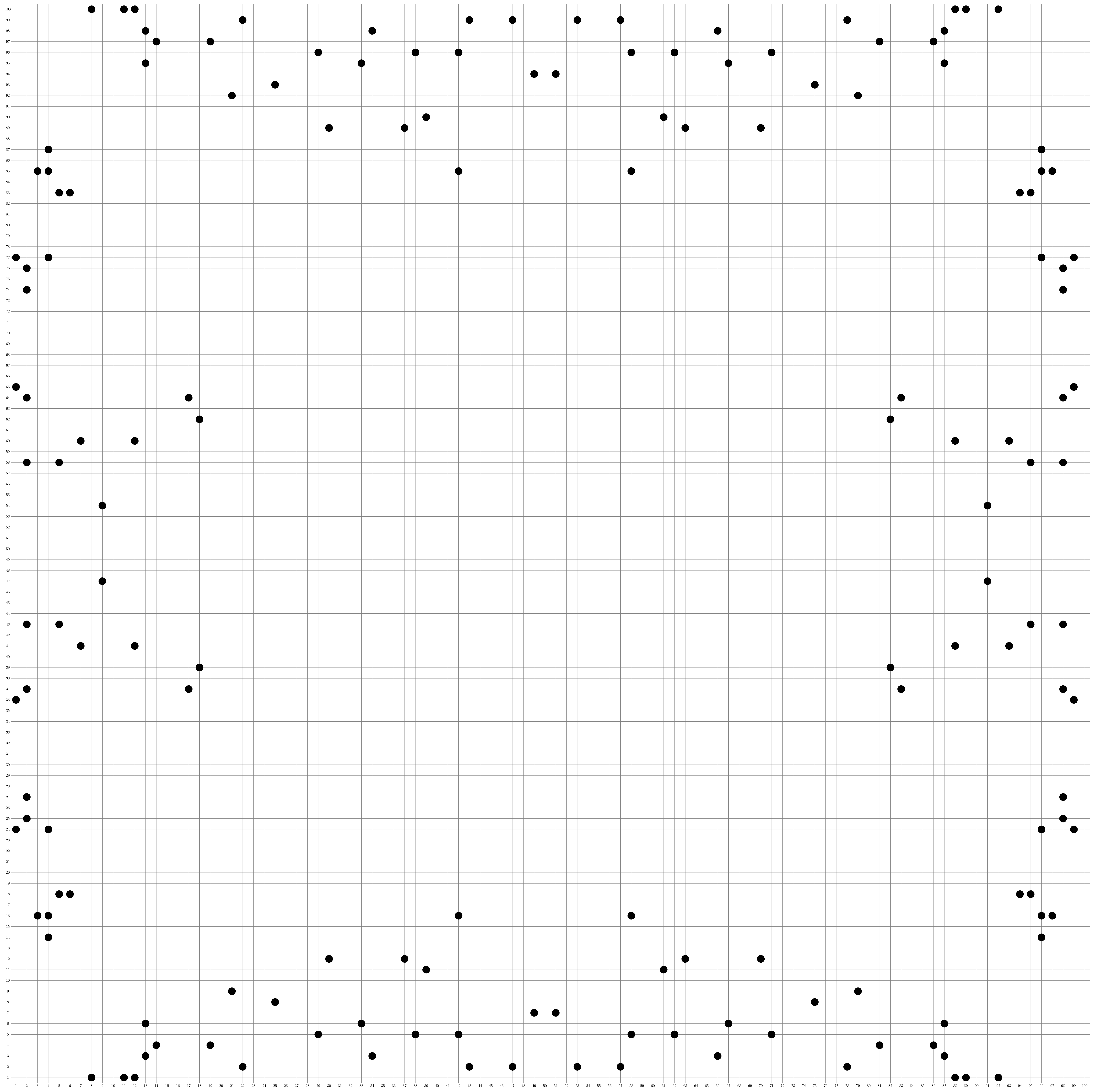}
    \caption{Left: a 154 point construction in $[100]^2$ found by standard methods. Right: a 160 point construction found by PatternBoost.}
    \label{fig:isosceles_100}
\end{figure}

\subsection{The ``no 5 on a sphere'' problem}

%\subsection{Subsets of $[n]^3$ with no 5 points on a sphere}

How many points can we choose from an $n \times n \times n$ grid, such that no 5 points lie on a sphere? Here we treat planes as degenerate spheres, so in particular 5 points on a plane are also not allowed. This is a another classic Erd\H{o}s problem. The answer is known~\cite{thiele1995geometric} to be between about $\frac{3}{16} \sqrt n$ and $4n$, so the right order of magnitude is not completely known. Can we repeat what we did in the problem of avoiding isosceles triangles in Section~\ref{subsec:isosceles}, and get good constructions for as many small values of $n$ as possible, and then guess whether the lower or the upper bound is more correct from the plot?

One of the reasons we were interested in this problem is that it is difficult to attack with traditional methods. Already, for small values of $n$ the number of constraints given by cospheral points is enormous. For example, for $n=4$ there are already over 700,000 5-tuples of cospheral points out of all $7.6$ million 5-tuples, and for $n=5$ the number of 5-tuples of points to check jumps to $230$ million! We tried with SAT solvers and linear programming methods, which work up to about $n=5$. However, even $n=6$ seems hopeless with these methods.

In contrast, this problem has an obvious setup for PatternBoost. We use a more naive local search than previously: given a string of proposed points $(x_1, x_2, \dots, x_m)$ with each $x_i \in [n]^3$ we try to add each point in succession, rejecting a new point if it is equal to a point which has already been added, or if it lies on a sphere containing four existing points. Once we have exhausted our list, we then try to add points randomly until addition of new points is impossible. Our PatternBoost loop is then the following:
\begin{enumerate}
\item Firstly, we generate $N$ configurations starting from the empty set (i.e. randomly adding points which don't lie on spheres through 4 existing points), and the top 10\% of constructions are kept.
\item These best constructions are tokenized, via $x_ i = (a_0, a_1, a_2) \mapsto a_0n^2 + a_1n + a_2$, so that each construction is represented\footnote{We use pythonic conventions: $[n] = \{ 0, 1, ..., n-1 \}$.} by a string of integers between $0$ and $n^3-1$. (Thus, in contrast to previous examples the tokenization is very simple.)
\item A transformer is trained on these tokenized constructions.
\item We then sample $N$ times from the transformer, and use these samples as seeds for local search as described above (i.e. we add try to add the points in specified order, and then resort to random search). We keep the top 10\% of the resulting constructions.
\item Steps 2-4 are repeated many (e.g. 10-50) times.
\end{enumerate}

Let us now point out a subtlety with this approach, which makes it initially difficult to carry out. Suppose our current construction consists of points $(x_0, \dots, x_{m-1})$ and we wish to decide whether a new point $x_m$ can be added. If we fix four points $x_i, x_j, x_k, x_l$, then our new point $x_m$ lies on the sphere that they span if and only if we have
\begin{equation} \label{eq:det}
\det \left (   \begin{matrix}% {c&c&c&c&c&c}
    x_{i,0} & x_{i,1} & x_{i,2} & x_{i,0}^2 + x_{i,1}^2 + x_{i,2}^2 & 1 \\
    x_{j,0} & x_{j,1} & x_{j,2} &x_{j,0}^2 + x_{j,1}^2 + x_{j,2}^2 & 1 \\
    x_{k,0} & x_{k,1} & x_{k,2} &x_{k,0}^2 + x_{k,1}^2 + x_{k,2}^2 & 1 \\
    x_{l,0} & x_{l,1} & x_{l,2} &x_{l,0}^2 + x_{l,1}^2 + x_{l,2}^2 & 1 \\
    x_{m,0} & x_{m,1} & x_{m,2} & x_{m,0}^2 + x_{m,1}^2 + x_{m,2}^2 & 1
    \end{matrix} \right ) = 0.
  \end{equation}
  Thus, in order to check whether we can add $x_m$ we need to evaluate $m \choose 4$ $5 \times 5$ determinants. Even in small examples (e.g. $n=6,7,8$) this is prohibitively slow, and does not produce enough examples in a reasonable amount of time. In order to get around this obstruction we apply two techniques:
  \begin{enumerate}
  \item Modern GPUs are optimized for many linear algebra computations in parallel. Thus, when set up correctly to run on GPU, we can evaluate the $m \choose 4$ determinants \eqref{eq:det} in a fraction of a second. (For example, generating $10K$ constructions for $n=4$ takes 500 seconds on a single CPU thread, and takes 23 seconds on GPU.)
  \item We exploit symmetry to augment our training data. Most constructions generated by local search will not have any symmetries, and hence we can turn 1 new construction into 48 new constructions ``for free'' by exploiting the symmetries of the cube.\footnote{One should be wary that data augmentation can lead to overfitting. This did not seem to be an issue in the examples we ran.}
  \end{enumerate}

Using these improvements, we ran PatternBoost for grid sizes $n$ up to $10$. Figure~\ref{fig:no5sphere} presents our best results, and Appendix \ref{app:sphere} includes a list of the best constructions we have found for $n=2,3,...,10$.% (see also Figure \ref{fig:no5sphere}).

  \begin{figure}[h!]
    \centering
    \includegraphics[width=0.6\linewidth]{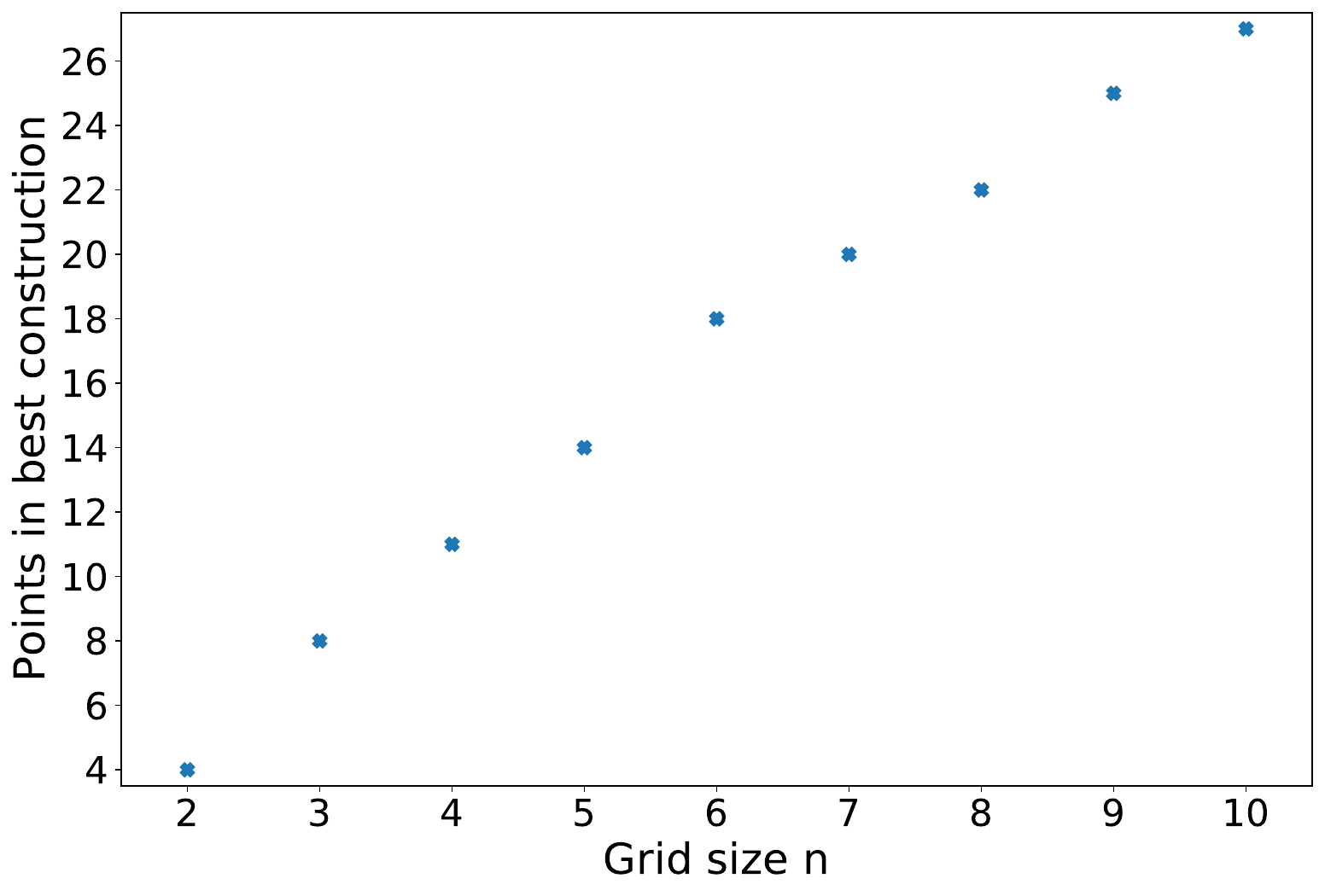}
        \caption{Best constructions for ``no 5 on a sphere'' for $n=2,3,\dots,10$, found by PatternBoost.}
    \label{fig:no5sphere}
\end{figure}

\subsection{Saturated Sperner systems of small size}
Imagine we are playing the following game: we are given an $n$ element set $\{1,2,\ldots,n\}$, and we keep picking subsets of it one after another. The only rule: we can never have chosen $k+1$ sets that are nested in each other (a $k+1$-chain). So for example when $k=2$, we cannot choose all three of the sets $\{1,2,5\}$, $\{1,2,3,5\}$, and $\{1,2,3,5,7,9\}$. We keep doing this until we cannot choose any more subsets without violating this condition. 
What is the earliest time we can get stuck?

\begin{figure}[h]
    \centering
    \begin{tikzpicture}[scale=0.7]

% Level 0
\node (empty) at (1.5,0) [circle, draw, fill=blue!30] {$\emptyset$};

% Level 1
\node (1) at (-1.5,2) [circle, draw, fill=blue!30] {1};
\node (2) at (0.5,2) [circle, draw, fill=blue!30] {2};
\node (3) at (2.5,2) [circle, draw] {3};
\node (4) at (4.5,2) [circle, draw] {4};

% Level 2
\node (12) at (-3.5,4) [circle, draw, fill=blue!30] {12};
\node (13) at (-1.5,4) [circle, draw] {13};
\node (14) at (0.5,4) [circle, draw] {14};
\node (23) at (2.5,4) [circle, draw] {23};
\node (24) at (4.5,4) [circle, draw] {24};
\node (34) at (6.5,4) [circle, draw, fill=blue!30] {34};

% Level 3
\node (123) at (-1.5,6) [circle, draw] {123};
\node (124) at (0.5,6) [circle, draw] {124};
\node (134) at (2.5,6) [circle, draw, fill=blue!30] {134};
\node (234) at (4.5,6) [circle, draw, fill=blue!30] {234};

% Level 4
\node (1234) at (1.5,8) [circle, draw, fill=blue!30] {1234};

% Edges from empty set to level 1
\draw (empty) -- (1);
\draw (empty) -- (2);
\draw (empty) -- (3);
\draw (empty) -- (4);

% Edges from level 1 to level 2
\draw (1) -- (12);
\draw (1) -- (13);
\draw (1) -- (14);

\draw (2) -- (12);
\draw (2) -- (23);
\draw (2) -- (24);

\draw (3) -- (13);
\draw (3) -- (23);
\draw (3) -- (34);

\draw (4) -- (14);
\draw (4) -- (24);
\draw (4) -- (34);

% Edges from level 2 to level 3
\draw (12) -- (123);
\draw (12) -- (124);

\draw (13) -- (123);
\draw (13) -- (134);

\draw (14) -- (124);
\draw (14) -- (134);

\draw (23) -- (123);
\draw (23) -- (234);

\draw (24) -- (124);
\draw (24) -- (234);

\draw (34) -- (134);
\draw (34) -- (234);

% Edges from level 3 to level 4
\draw (123) -- (1234);
\draw (124) -- (1234);
\draw (134) -- (1234);
\draw (234) -- (1234);

\end{tikzpicture}

    \caption{A saturated 4-Sperner family of size 8. The addition of any set creates a chain of length 5.}
    \label{fig:enter-label}
\end{figure}
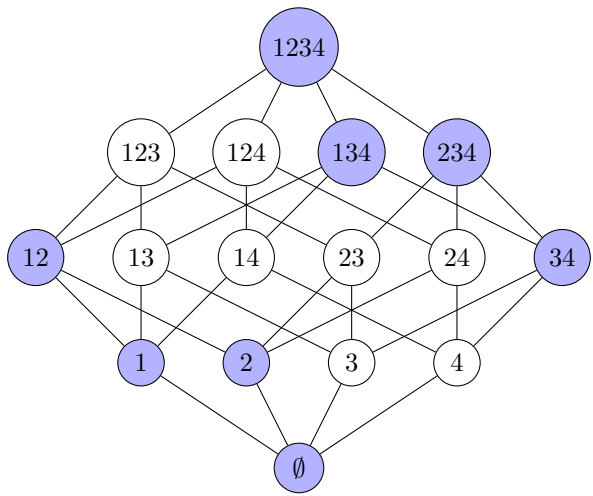

Formally, a family $\mathcal{F}$ of subsets of $\{1,2,\ldots,n\}$ is called \emph{$k$-Sperner} if it doesn't contain $k+1$ distinct sets $A_1,A_2,\ldots,A_{k+1}$ satisfying $A_1\subset A_2\subset \ldots \subset A_{k+1}$.  The family $\mathcal{F}$ is called \emph{saturated $k$-Sperner} if it is $k$-Sperner, and adding any set to $\mathcal{F}$ that is not already in $\mathcal{F}$ would destroy this property, i.e.~if the addition of any new set would create a chain of length $k+1$. With this terminology, the question is: what is the smallest possible size of a saturated $k$-Sperner family?

In~\cite{gerbner2013saturating}, Gerbner et al.~conjectured that if $n$ is sufficiently large compared to $k$, then the minimum size of a saturated $k$-Sperner system  is $2^{k-1}$. There is a simple construction achieving this, and they showed that this is optimal for $k=1,2,3$. Recently, Morrison, Noel, and Scott~\cite{morrison2014saturated} showed that their conjecture is true for all $k\leq 5$ and disproved this conjecture for larger $k$. They showed that there exists an $\epsilon > 0$ such that, for all $k$, and for all $n$ sufficiently large compared to $k$, there exists a saturated $k$-Sperner family of size $2^{(1-\epsilon)k}$. Their $\epsilon$ was approximately $0.023$, which they obtained by constructing a saturated 6-Sperner system of cardinality 30. This was recently improved to $\epsilon=0.0385$ by Martin--Veldt~\cite{martin2024saturation}, where they constructed a saturated 7-Sperner system of size 56. With PatternBoost we were able to construct a saturated 8-Sperner system of size 108, see Figure~\ref{fig:sper_12} and~Appendix~\ref{app:saturated_sperner}, which with the same methods as in~\cite{morrison2014saturated}, yields an improved bound $\epsilon = 0.04085$.

\begin{figure}[h!]
    \centering
    \includegraphics[width=\linewidth]{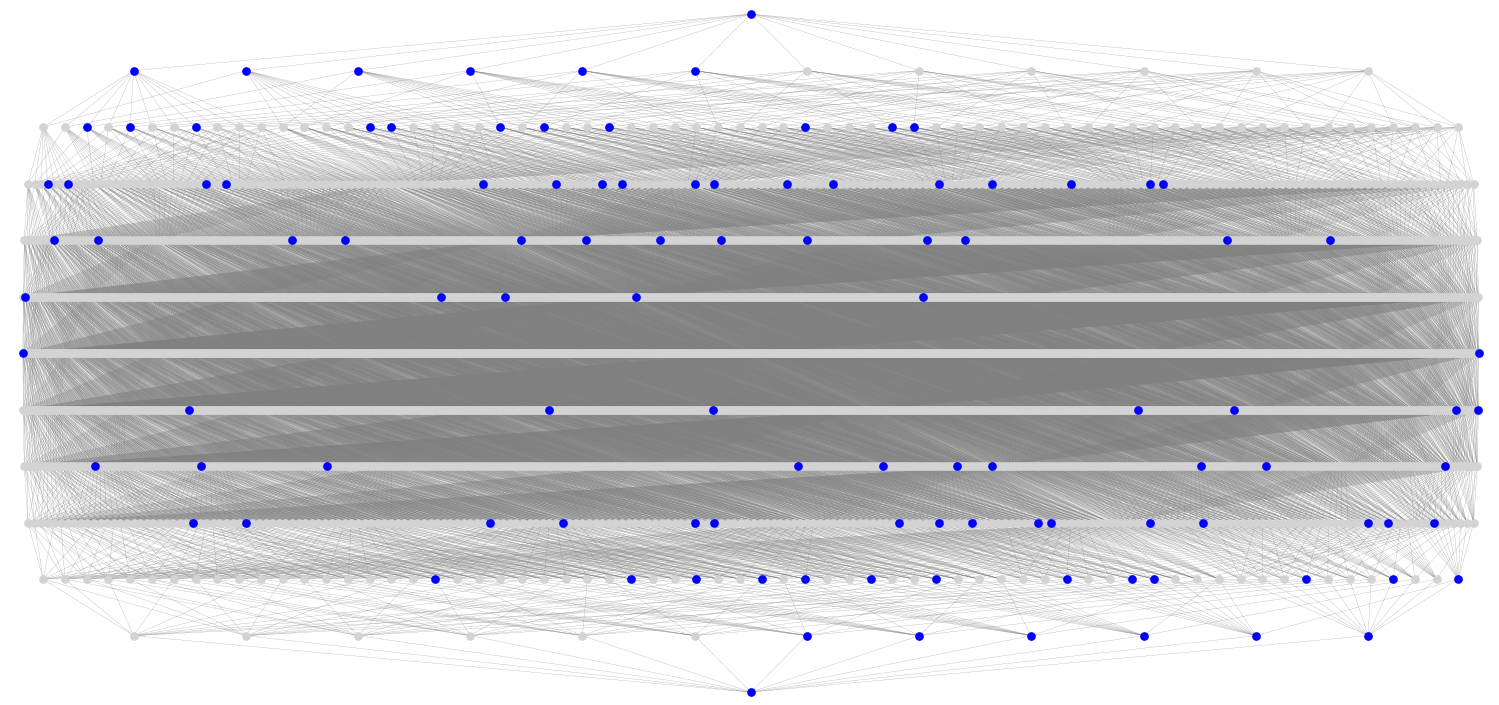}
    \caption{A saturated 8-Sperner family of size 108, with $n=12$. The addition of any set creates a chain of length 9. The exact sets can be found in~Appendix~\ref{app:saturated_sperner}.}
    \label{fig:sper_12}
\end{figure}

This problem does not have an obvious choice for a reward function. We used the structural characterization that was described in~\cite{morrison2014saturated} to come up with an ad hoc score function, and experimented with various parameters until it worked reasonably well. It is likely that through a better understanding of the problem, and the new insights in~\cite{martin2024saturation}, further improvements to this result are possible.

\subsection{Cross-Sperner systems}

Let $\mathcal{F}_1,\mathcal{F}_2,\ldots ,\mathcal{F}_k$ be $k$ families of subsets of $\{1,2,\ldots,n\}$. The tuple  $(\mathcal{F}_1,\mathcal{F}_2,\ldots ,\mathcal{F}_k)$ is called \emph{cross-Sperner} if there is no pair $i\neq j$ for which some $F_i \in \mathcal{F}_i$ is a subset of some $F_j \in \mathcal{F}_j$. How large can $\prod_{i=1}^k |\mathcal{F}_i|$ be, as a function of $n$ and $k$? Let us denote the answer by $\pi (n,k)$.

\begin{figure}[h]
    \centering
\begin{tikzpicture}[scale=0.525]

% Level 0
\node (empty) at (1.5,0) [circle, draw] {$\emptyset$};

% Level 1
\node (1) at (-1.5,2) [circle, draw, fill=blue!30] {1};
\node (2) at (0.5,2) [circle, draw] {2};
\node (3) at (2.5,2) [circle, draw] {3};
\node (4) at (4.5,2) [circle, draw] {4};

% Level 2
\node (12) at (-3.5,4) [circle, draw, fill=blue!30] {12};
\node (13) at (-1.5,4) [circle, draw, fill=blue!30] {13};
\node (14) at (0.5,4) [circle, draw] {14};
\node (23) at (2.5,4) [circle, draw] {23};
\node (24) at (4.5,4) [circle, draw] {24};
\node (34) at (6.5,4) [circle, draw] {34};

% Level 3
\node (123) at (-1.5,6) [circle, draw, fill=blue!30] {123};
\node (124) at (0.5,6) [circle, draw] {124};
\node (134) at (2.5,6) [circle, draw] {134};
\node (234) at (4.5,6) [circle, draw] {234};

% Level 4
\node (1234) at (1.5,8) [circle, draw] {1234};

% Edges from empty set to level 1
\draw (empty) -- (1);
\draw (empty) -- (2);
\draw (empty) -- (3);
\draw (empty) -- (4);

% Edges from level 1 to level 2
\draw (1) -- (12);
\draw (1) -- (13);
\draw (1) -- (14);

\draw (2) -- (12);
\draw (2) -- (23);
\draw (2) -- (24);

\draw (3) -- (13);
\draw (3) -- (23);
\draw (3) -- (34);

\draw (4) -- (14);
\draw (4) -- (24);
\draw (4) -- (34);

% Edges from level 2 to level 3
\draw (12) -- (123);
\draw (12) -- (124);

\draw (13) -- (123);
\draw (13) -- (134);

\draw (14) -- (124);
\draw (14) -- (134);

\draw (23) -- (123);
\draw (23) -- (234);

\draw (24) -- (124);
\draw (24) -- (234);

\draw (34) -- (134);
\draw (34) -- (234);

% Edges from level 3 to level 4
\draw (123) -- (1234);
\draw (124) -- (1234);
\draw (134) -- (1234);
\draw (234) -- (1234);

\end{tikzpicture}
\begin{tikzpicture}[scale=0.525]

% Level 0
\node (empty) at (1.5,0) [circle, draw] {$\emptyset$};

% Level 1
\node (1) at (-1.5,2) [circle, draw] {1};
\node (2) at (0.5,2) [circle, draw] {2};
\node (3) at (2.5,2) [circle, draw] {3};
\node (4) at (4.5,2) [circle, draw, fill=blue!30] {4};

% Level 2
\node (12) at (-3.5,4) [circle, draw] {12};
\node (13) at (-1.5,4) [circle, draw] {13};
\node (14) at (0.5,4) [circle, draw] {14};
\node (23) at (2.5,4) [circle, draw] {23};
\node (24) at (4.5,4) [circle, draw, fill=blue!30] {24};
\node (34) at (6.5,4) [circle, draw, fill=blue!30] {34};

% Level 3
\node (123) at (-1.5,6) [circle, draw] {123};
\node (124) at (0.5,6) [circle, draw] {124};
\node (134) at (2.5,6) [circle, draw] {134};
\node (234) at (4.5,6) [circle, draw, fill=blue!30] {234};

% Level 4
\node (1234) at (1.5,8) [circle, draw] {1234};

% Edges from empty set to level 1
\draw (empty) -- (1);
\draw (empty) -- (2);
\draw (empty) -- (3);
\draw (empty) -- (4);

% Edges from level 1 to level 2
\draw (1) -- (12);
\draw (1) -- (13);
\draw (1) -- (14);

\draw (2) -- (12);
\draw (2) -- (23);
\draw (2) -- (24);

\draw (3) -- (13);
\draw (3) -- (23);
\draw (3) -- (34);

\draw (4) -- (14);
\draw (4) -- (24);
\draw (4) -- (34);

% Edges from level 2 to level 3
\draw (12) -- (123);
\draw (12) -- (124);

\draw (13) -- (123);
\draw (13) -- (134);

\draw (14) -- (124);
\draw (14) -- (134);

\draw (23) -- (123);
\draw (23) -- (234);

\draw (24) -- (124);
\draw (24) -- (234);

\draw (34) -- (134);
\draw (34) -- (234);

% Edges from level 3 to level 4
\draw (123) -- (1234);
\draw (124) -- (1234);
\draw (134) -- (1234);
\draw (234) -- (1234);

\end{tikzpicture}
  \caption{A cross-Sperner family for the parameters $(n,k)=(4,2)$. We have $|\mathcal{F}_1|\cdot|\mathcal{F}_2|=4\cdot 4=16$, which is optimal. So $\pi(4,2)=16$.}
    \label{fig:enter-label}
\end{figure}
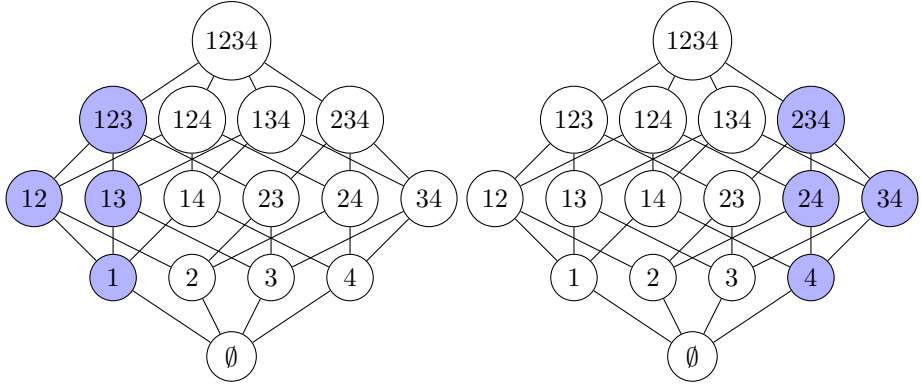

This problem was recently considered by Behague, Kuperus, Morrison, and Wright~\cite{behague2023improved}. They gave constructions that provide a lower bound on $\pi(n,k)$ for all pairs $(n,k)$, and they stated a general conjecture for an upper bound as well. If this conjecture is true, then it would be tight for $n$ large enough. 

We can set this problem up for PatternBoost in a simple way. Given a collection of families, we determine their score by deleting sets from them that would violate the cross-Sperner condition, and then calculating the products of the sizes of the smaller set families. With this setup, we were able to beat the constructions from~\cite{behague2023improved} for several small values of $(n,k)$, see Table~\ref{tab:cross_sperner} for the results and Appendix~\ref{app:cross_sperner} for the constructions. We did not manage to beat their conjectured upper bound for any values of the parameters. To us, it seems likely that the conjecture is true.

\begin{table}[h!]
\centering
\begin{tabular}{|c|c|m{3.3cm}|c|}
\hline
$(n,k)$ & BKMW construction &\centering BKMW conjectured \newline upper bound & Our construction \\
\hline
 $(6,3)$ & 810  &\centering 852  &  810 \\
\hline
 $(7,3)$&6075  &\centering 6818& \textbf{6480} \\
\hline
 $(8,3)$&50625  &\centering 54551 & \textbf{51840} \\
\hline
 $(9,3)$& 421875 &\centering 436413 &414720  \\
\hline $(5,4)$& 108 &\centering 129  & 108  \\
\hline $(6,4)$&  729&\centering 2075  & \textbf{1764} \\
\hline $(7,4)$& 19683  &\centering 33215  &  \textbf{28350}\\
\hline $(8,4)$& 531441  &\centering 531441  &531441  \\
\hline
\end{tabular}
\caption{Comparison of constructions and conjectured upper bounds for various $(n,k)$ values. The bold values denote an improvement over the previous best known constructions.}
\label{tab:cross_sperner}
\end{table}

\subsection{Double covers of high dimensional boxes}

The following elegant problem was first posed by Kearnes and Kiss~\cite{kearnes1999finite} during the open problem session at the August 1999 MIT meeting held in honor of Daniel Kleitman’s 65th birthday~\cite{saks2002kleitman}. A set of the form $A=A_1\times A_2\times\ldots\times A_d$, where $A_1, A_2, \ldots, A_d$ are finite sets with $|A_i| \geq 2$, is referred to here as a \emph{$d$-dimensional discrete box}. A set of the form $B = B_1 \times B_2 \times \ldots \times B_d$, where  $B_i \subseteq A_i$ for all $i \in [d]$, is a sub-box of $A$. A sub-box $B$ is considered proper if $B_i \neq A_i$ for every $i$. Kearnes and Kiss posed the following question: can the box $A = A_1 \times A_2 \times\ldots\times A_d$ be partitioned into fewer than $2^d$ proper sub-boxes?

Within a day, Alon, Bohman, Holzman and Kleitman solved~\cite{alon2002partitions} this problem, by proving that a $d$-dimensional discrete box cannot be partitioned into fewer than $2^d$ proper sub-boxes. Their eventual distillation of the proof is a real gem in mathematics. As is often the case, their theorem opened up further interesting questions, many of which are still unsolved.

Here we will attempt to make progress on the following problem, due to Leader--Mili\'cevi\'c--Tan~\cite{leader2018decomposing}.

\begin{question}\label{que:imrebox}
    Let $A$ be a $d$-dimensional discrete box, and suppose $B^{(1)},B^{(2)},\ldots,B^{(m)}$ are proper sub-boxes that cover every point in $A$ exactly twice. Must we have $m\geq 2^d$?
\end{question}

The earlier theorem tells us that to cover every point exactly once, we need at least $2^d$ sub-boxes. Intuitively, to cover every point twice means we will have to use more boxes. On the other hand, we now have more flexibility as our boxes need not be disjoint anymore.

Problems of this type seem to be a good fit for PatternBoost. There are many ways to pick a local search and a scoring function, and we have little idea what works best. We fixed the box $A$ as $\{0,1,2\}^d$ for some $d$, and used various simple ad hoc heuristics in the local search function. For score functions, we penalized covering a point more than twice more than covering a point less than twice, as any set of boxes that cover every point at most twice can be more easily extended to a correct solution by adding size 1 boxes to fill in the holes. Our results, which match the numbers we could get by applying the LP solver Gurobi~\cite{gurobi}, can be seen in Table~\ref{tab:boxes}.

\begin{table}[h!]
\centering
\begin{tabular}{|c|c|m{3.3cm}|c|}
\hline
$A$ & m  \\
\hline
 $\{0,1,2\}^2$ & 6   \\
 $\{0,1,2\}^3$ & 11   \\
 $\{0,1,2\}^4$ & 21   \\
 $\{0,1,2\}^5$ & 41   \\
 \hline
\end{tabular}
\caption{The smallest number $m$ of proper sub-boxes we could find to cover every point of the box $A=\{0,1,2\}^d$ exactly twice.}
\label{tab:boxes}
\end{table}

Our construction for $\{0,1,2\}^5$ can be found in Appendix~\ref{app:boxes}. For $\{0,1,2\}^6$ our result implies that there is a construction with $2\cdot 41$ sub-boxes, but neither Gurobi nor PatternBoost succeeded in coming up with a construction using fewer than $82$ boxes within a reasonable time. In this case we did not even get close with PatternBoost, it is likely that a smarter local search method is needed to make further progress. Hence the best result we were able to prove is the following:

\begin{theorem}
    Let $d\geq 5$. There is a collection of $\frac{41}{32}2^d\approx 1.28\cdot 2^d$ proper sub-boxes that cover every point of $\{0,1,2\}^d$ exactly twice.
\end{theorem}

It would be fascinating to find a way to explore even higher dimensions, to reduce this constant $1.28$ even more. Pushing this constant below 1 would resolve Question~\ref{que:imrebox}.

\section{Machine learning settings}

In this section, we present the main components of our machine learning architecture. We first describe the model, and the training and generation procedure, then we discuss the tokenization schemes used to represent mathematical constructs.

Makemore~\cite{makemore} is an implementation of GPT-2~\cite{radford2019language}, a decoder-only transformer model. It processes a sequence $t_1 t_2 t_3 \dots t_l$ of $l$ tokens (words) from a given vocabulary of size $v$ and outputs the conditional probabilities of the tokens in positions $2$ to $l+1$, given the corresponding prefix, i.e. $p(t_2|t_1), p(t_3|t_1t_2), \dots p(t_{l+1}|t_1t_2\dots t_l)$. In practice, we only care about $p(t_{l+1})$, the probabilities of the ``next token''.

In the model, the input tokens, encoded as one hot vectors (vector of $v$ integers, all zero except the position of $t_i$ in the vocabulary), are first transformed into vectors in $\mathbb  R^d$ by multiplying them by a learnable matrix of size $v\times d$, the Embedding.
In similar fashion, the position of each token, $i$ for token $t_i$, is also transformed into a vector in $\mathbb R^d$ by a Positional Embedding (a $v\times d$ matrix, also learnable). Token and Position embeddings are then added. At this point, the input is encoded as a sequence $w_1 \dots w_l$ of $l$ vectors in $\mathbb R^d$. This sequence is then processed by several \emph{transformer layers}, which transform it by adding correction factors, and performing normalizations.

In each layer, the sequence is first processed by a \emph{self-attention layer}, which computes the dot products between linear transformations of the input at different positions, i.e. products of the form $V((Kw_i) \cdot (Qw_j))$, with $V$, $K$ and $Q$ learnable matrices. The result of this calculation is added to the input. 
The resulting sequence is then processed by a \emph{fully-connected neural network}, with one hidden layer of $4d$ neurons, and GeLU activations~\cite{hendrycks2023gelu}. The result of this computation is added to the input sequence, which is then passed to the next transformer layer.

Finally, the output of the last transformer layer, $o_1 o_2\dots o_l$ ($o_i \in \mathbb R^d$) is decoded into next token probabilities, by multiplying it by a $d\times v$ matrix $D$, $z_i=Do_i$, for $i\leq l$, and applying to each $z_i$ a softmax operator ($\sigma(z_{i,j})=e^{z_{i,j}}/{\sum_{k=1}^v e^{z_{i,k}}}$), to transform $z_i$ into a vector of real numbers between $0$ and $1$, summing to $1$: the next token probabilities. In this paper, we use transformers with $2$ to $6$ layers, a dimension $d$ between $16$ and $128$, and $4$ attention heads.

During training, the model is presented with sequences of $l$ words representing mathematical constructs, prefixed by an initial token $t_0$, and computes the next token probabilities for all prefixes: from $t_0$ to the full sequence. We calculate the cross-entropy loss, $\sum_{i=0}^l {-\log(p(t_{i+1},i+1))}$, where $p(x,i)$ is the probability predicted by the model for the token $x$ in position $i$, and $t_{l+1}$ is a special end-of-sentence token. This rewards the model for assigning high probabilities to the actual tokens $t_i$, thus learning the specifics of the input sequence. The derivatives of the cross-entropy loss,  with respect to models parameters, are calculated using back-propagation, and accumulated. Every $32$ examples, model parameters are updated using the optimizer AdamW with a learning rate of $5 \cdot 10^{-4}$, and a weight decay of $0.01$.
The model is trained for a fixed number of optimization steps, chosen so that the model does not overfit its training data (i.e. the loss over the training sample remains close to the loss over an external test set).

Once the model is trained, we generate candidate solutions, starting from an ``empty sequence'' $t_0$, and sample a first token $t_1$ from the distribution predicted by the trained model. Then, we present the model with the sequence $t_0 t_1$, and sample the second token $t_2$ from the predicted distribution. We repeat this until a full solution is generated. This encourages the model to generate sequences similar to its training data (i.e. promising solutions to our problem).

To be processed by the transformer, mathematical constructs must be encoded as sequences of tokens using some fixed vocabulary, a process known as \emph{tokenization}. Grids and graphs can be represented as sequences of zeros and ones ($n^2$ binary entries for an $n \times n$ grid, $n(n-1)/2$ entries in an adjacency matrix for a graph with $n$ vertices). However, this naive representation, which uses a vocabulary of size $v=2$, results in very long sequences even for modest values of $n$. This complicates learning: the attention mechanism used by transformers scales, both in speed and memory usage, as the square of the length of the input, roughly $n^4$ and $n^4/4$ with this naive encoding. To improve learning efficiency, we encode several binary entries in a single token, trading a larger vocabulary for shorter sequences. By convention, sequences always begin and end with special ``\textless start\textgreater'' and ``\textless end\textgreater'' tokens.

For this purpose, we use fixed and variable length tokenizers. In the former, groups of $k$ binary entries are encoded as integers between $0$ and $2^k-1$ (e.g. binary matrices are encoded $8$ entries at a time, as integers between $0$ and $255$).
In the latter, special tokens are introduced, to represent commonly occurring binary sequences (e.g. long sequences of zeroes if we deal with sparse matrices). This results in even shorter sequences. In this paper, we use byte pair encoding (BPE), a common tokenization method for natural language. We start with a vocabulary of two ($0$ and $1$), and a set of representative model inputs $\mathcal S$, and iteratively introduce new tokens to represent the most common subsequences in $\mathcal S$. For instance, given $\mathcal S=\{100001,110001, 001001 \}$, we observe that the most common subsequence in $\mathcal S$ is $00$, and encode it as $2$, yielding $\mathcal S=\{1221,11201, 2121 \}$. The most common subsequence is now $12$, which we encode as $3$, yielding $\mathcal S=\{321,1301, 231 \}$. At this point, the average length of the sequences in $\mathcal S$ have been reduced from $6$ to $3.33$. The process is repeated until a fixed vocabulary size is reached.

\section{Concluding remarks}

In this paper, we introduced PatternBoost, a versatile machine learning algorithm designed to uncover ``good'' mathematical constructions. It seems to be perhaps most useful when the best constructions to the problem clearly follow some pattern, but these patterns are too complicated for us to fully understand. A key advantage of PatternBoost is its accessibility -- it can be used effectively by mathematicians without deep machine learning expertise, making it a practical tool in a wide range of mathematical applications. The method relies on two core components, which together determine how well it will perform. The first is a local search algorithm that takes an object as an input and improves it by making small local modifications. The second is a global learning algorithm, where a transformer model is trained on the best constructions to generate new starting points for the local search. We demonstrated that even with the simplest local search algorithm, PatternBoost can achieve state-of-the-art results in multiple cases. In practice however, one should always tailor the local search algorithm to the specific problem one wants to solve, as this can significantly enhance the performance of PatternBoost.

 Looking ahead, we believe that methods like PatternBoost have the potential to become a standard tool for researchers seeking to combine machine learning with mathematical discovery. PatternBoost's flexibility allows it to be adapted to a wide variety of problems, and its ease of use will hopefully lower the barrier for mathematicians to engage with machine learning techniques. As more problem-specific adaptations of the local search component are developed, we anticipate that PatternBoost will continue to push the boundaries of what can be achieved in automated mathematical construction. We are optimistic that this approach will not only help uncover new mathematical structures but also inspire further collaboration between the fields of mathematics and artificial intelligence.

\section*{Acknowledgements}

The authors would like to thank the Center Of Mathematical Sciences And Applications at Harvard University for running the inspiring program on ``Mathematics and Machine Learning'', where some of this research was done. GW would also like to thank Alex Davies, AZW would like to thank Gwena\"el Joret, and all authors  would like to thank Marijn Heule for valuable discussions.

\bibliographystyle{plain}
\bibliography{refs}

\appendix

\section{A proof of a lower bound on sets without isosceles triangles}\label{app:iso_proof}

Recall from Subsection~\ref{subsec:isosceles} the definition
$$f(n) := \max_{S\subset [n]^2}\{|S|: a,b,c\in S \text{ distinct} \implies d(a,b) \neq d(b,c)\},$$where $d(a,b)$ denotes Euclidean distance.

\begin{proposition}
    There exists an absolute constant $c>0$ such that $$f(n)\geq \frac{cn}{\sqrt{\log n}}.$$
\end{proposition}
\begin{proof}
    Let $S$ be a random subset of the $n \times n$ grid $[n]^2$, obtained by independently choosing each point to be in $S$ with probability $p$ to be determined later. Let $v=(v_x,v_y)$ be an arbitrary point in $[n]^2$, and denote by $d_k(v)$ the set of points $u=(u_x,u_y)\in [n]^2$ for which $d(u,v)^2=(u_x-v_x)^2 + (u_y-v_y)^2 = k$. 

    What is the probability that $v$ is in $S$, and it is the ``peak'' of an isosceles triangle in $S$, i.e.~that there exist two other points $u,w$ with $d(v,u)=d(v,w)$? This would mean that for some integer $k$, we have $|S\cap d_k(v)|\geq 2$. This can be bounded by repeatedly using the union bound, as
    \begin{equation*}
        \mathbb{P}(\exists k: |S\cap d_k(v)|\geq 2)\leq \sum_{k=1}^{2n^2}\mathbb{P}(|S\cap d_k(v)|\geq 2)\leq p^2\sum_{k=1}^{2n^2} |d_k(v)|^2\leq Cp^2n^2\log n.
    \end{equation*}
    Here the last inequality, where $C>0$ is an absolute constant, follows from the asymptotics of the second moment of the sum of squares function~\cite{numbertheorybound,blomer2006estimates}. This means that the probability that a particular vertex $v$ is the peak of an isosceles triangle in $S$ is
    \begin{equation*}
    \begin{split}
        \mathbb{P}(v\text{ is the peak of an isosceles triangle in }S)&= 
        \mathbb{P}(v\in S)\mathbb{P}(\exists k: |S\cap d_k(v)|\geq 2)\\ &\leq Cp^3n^2\log n.
        \end{split}
    \end{equation*}
    Therefore, the expected number of points in $[n]^2$ that are peaks of isosceles triangles in $S$ is at most $Cp^3n^4\log n$. Let $P_S\subset S$ denote these vertices. Note that the set $S\setminus P_S$ contains no isosceles triangles, and has expected size
    $$\mathbb{E}(|S\setminus P_S|)=\mathbb{E}(|S|)-\mathbb{E}(|P_S|)\geq pn^2-Cp^3n^4\log n.$$ Setting $p=\frac{\epsilon}{n\sqrt{\log n}}$, where $\epsilon>0$ is a sufficiently small constant, yields
    $$\mathbb{E}(|S\setminus P_S|)\geq \frac{n}{\sqrt{\log n}}\left(\epsilon - C\epsilon^3\right)=: \epsilon' \frac{n}{\sqrt{\log n}}.$$
    Therefore, there exists a choice of $S$ for which $|S\setminus P_S|\geq \epsilon'\frac{n}{\sqrt{\log n}}$, finishing the proof.
\end{proof}

\section{Subsets of $[n]^3$ with no 5 points on a sphere}\label{app:sphere}

$n=3,$ 8 points:
(1, 1, 3), (1, 2, 1), (1, 3, 3), (2, 2, 2), (2, 2, 3), (2, 3, 1), (3, 1, 1), (3, 3, 2)

~

\noindent$n=4,$ 11 points: (1, 1, 3), (1, 3, 2), (1, 4, 1), (2, 1, 1), (2, 1, 3), (2, 2, 3), (2, 3, 4), (3, 4, 2), (3, 4, 4), (4, 2, 4), (4, 4, 1)

~

\noindent
$n=5,$ 14 points: (1, 1, 3), (1, 2, 1), (1, 4, 5), (1, 5, 3), (2, 2, 5), (2, 3, 2), (2, 5, 2), (3, 2, 4), (4, 3, 1), (4, 4, 4), (4, 5, 5), (5, 1, 1), (5, 2, 4), (5, 5, 4)

~

\noindent
%$n=6$, 17 points: (1, 2, 1), (1, 3, 6), (1, 5, 6), (2, 1, 6), (2, 3, 3), (2, 6, 2), (3, 1, 1), (3, 2, 5), (3, 6, 2), (4, 1, 2), (4, 3, 2), (4, 5, 5), (4, 6, 5), (5, 3, 1), (6, 1, 5), (6, 4, 6), (6, 6, 1)
$n=6$, 18 points: (1, 1, 2), (1, 4, 1), (1, 5, 3), (1, 5, 5), (2, 1, 5), (2, 2, 3), (2, 4, 6), (2, 6, 6), (3, 5, 2), (3, 6, 1), (4, 1, 6), (4, 3, 4), (4, 6, 5), (5, 2, 2), (6, 1, 1), (6, 2, 1), (6, 3, 5), (6, 5, 6)

~

\noindent
$n=7,$ 20 points: (1, 1, 1), (1, 3, 3), (1, 6, 7), (1, 7, 6), (2, 1, 5), (2, 2, 4), (2, 6, 1), (3, 6, 4), (4, 2, 1), (4, 7, 2), (5, 1, 6), (5, 4, 4), (5, 5, 3), (6, 1, 7), (6, 3, 6), (6, 5, 7), (7, 2, 1), (7, 5, 2), (7, 6, 2), (7, 7, 7)

~

\noindent
$n=8,$ 22 points: (1, 1, 2), (1, 2, 5), (1, 6, 8), (1, 7, 1),
 (2, 7, 2), (2, 8, 6), (3, 1, 5), (3, 3, 8),
 (3, 7, 6), (4, 1, 8), (4, 3, 1), (5, 3, 3),
 (5, 4, 4), (5, 5, 7), (6, 7, 4), (6, 8, 1),
 (6, 8, 8), (7, 2, 1), (7, 5, 6), (8, 1, 3),
 (8, 6, 7), (8, 8, 5)

~

\noindent
$n=9,$ 25 points: (1, 1, 9), (1, 2, 1), (1, 4, 1), (1, 9, 4), (2, 1, 6), (2, 4, 5), (2, 6, 8), (2, 9, 9), (3, 3, 9), (3, 8, 3), (4, 2, 3), (4, 5, 2), (4, 9, 8), (5, 2, 2), (5, 8, 2), (6, 1, 3), (6, 4, 8), (6, 5, 5), (7, 7, 6), (8, 1, 1), (8, 2, 9), (8, 3, 6), (9, 7, 1), (9, 8, 7), (9, 9, 4)

~

\noindent
$n=10,$ 27 points: (1, 1, 7), (1, 2, 3), (1, 9, 9), (2, 2, 5), (2, 3, 9), (2, 4, 1), (3, 1, 2), (3, 5, 2), (3, 5, 4), (3, 9, 5), (4, 3, 5), (4, 6, 3), (5, 1, 3), (5, 7, 1), (7, 7, 3), (8, 2, 7), (8, 7, 1), (9, 5, 5), (10, 1, 1), (10, 3, 9), (11, 1, 4), (11, 2, 2), (11, 7, 6), (12, 3, 7), (12, 9, 1), (13, 1, 8), (13, 3, 3)

\section{A saturated 8-Sperner system of size 108}\label{app:saturated_sperner}

The set family is the union of all sets below. We partitioned it according to the structural characterization of~\cite{morrison2014saturated} for clarity.

~

{\footnotesize
$\mathcal{A}_0 = \{\emptyset\}$}

~

{\footnotesize
$\mathcal{A}_1$=$\{\{1\}, \{2\}, \{3\}, \{4\}, \{5\}, \{6\}, \{7, 8, 9, 10, 11, 12\}\}$}

~

{\footnotesize
$\mathcal{A}_2=\{
\{1, 4\}, \{3, 4\}, \{1, 6\}, \{3, 6\}, \{2, 7\}, \{5, 7\}, \{2, 8\}, \{5, 8\}, \{1, 9\}, \{3, 9\}, \{4, 10\},$ $\{1, 2, 3, 5, 10, 11, 12\},$ $ \{1, 3, 7, 8, 10, 11, 12\},$ $ \{2, 4, 5, 6, 9, 11, 12\},$ $ \{2, 5, 6, 9, 10, 11, 12\},$ $ \{4, 6, 7, 8, 9, 11, 12\},$ $ \{6, 7, 8, 9, 10, 11, 12\}\}$}

~

{\footnotesize
$\mathcal{A}_3=
\{\{1, 2, 6\}, \{3, 4, 6\}, \{1, 5, 6\}, \{3, 6, 7\},$ $ \{4, 5, 8\},$ $ \{2, 6, 8\},$ $ \{5, 7, 8\}, $ $\{1, 2, 9\},$ $ \{3, 4, 9\},$ $ \{1, 5, 9\},$ $ \{3, 7, 9\},$ $ \{2, 8, 9\},$ $ \{2, 4, 10\},$ $ \{4, 6, 10\},$ $ \{2, 7, 10\},$ $ \{5, 7, 10\},$ $ \{4, 9, 10\},$ $ \{1, 2, 3, 4, 5, 7, 11, 12\},$ $ \{1, 2, 3, 4, 7, 8, 11, 12\},$ $ \{1, 2, 3, 5, 8, 10, 11, 12\},$ $ \{1, 3, 4, 5, 10, 11, 12\},$ $ \{1, 3, 4, 7, 8, 10, 11, 12\},$ $ \{1, 3, 6, 8, 9, 10, 11, 12\},$ $ \{1, 4, 6, 7, 8, 9, 11, 12\},$ $ \{1, 6, 7, 8, 9, 10, 11, 12\}, $ $\{2, 3, 5, 6, 9, 10, 11, 12\}, $ $\{2, 4, 5, 6, 7, 9, 11, 12\},$ $ \{3, 5, 6, 8, 9, 10, 11, 12\}\}$}

~

{\footnotesize
$\mathcal{A}_4=$
$\{\{1, 2, 4, 6\}, \{1, 5, 6, 7\}, \{1, 2, 6, 8\}, \{2, 3, 4, 9\}, \{3, 5, 7, 9\}, \{2, 3, 8, 9\},$ $ \{1, 4, 6, 10\},$ $ \{2, 4, 7, 10\},$ $ \{2, 5, 7, 10\},$ $ \{2, 7, 8, 10\},$ $ \{3, 4, 9, 10\},$ $ \{4, 7, 9, 10\},$ $ \{5, 7, 9, 10\},$ $ \{1, 2, 3, 4, 5, 7, 8, 11, 12\},$ $ \{1, 2, 3, 4, 5, 8, 10, 11, 12\},$ $ \{1, 2, 3, 5, 6, 9, 10, 11, 12\}, $ $\{1, 2, 3, 6, 7, 9, 10, 11, 12\},$ $ \{1, 2, 4, 5, 7, 8, 9, 11, 12\},$ $ \{1, 2, 4, 5, 8, 9, 10, 11, 12\}, $ $\{1, 3, 4, 5, 6, 8, 9, 11, 12\},$ $ \{1, 3, 4, 5, 7, 8, 10, 11, 12\},$ $ \{1, 3, 4, 6, 7, 8, 9, 11, 12\},$ $ \{1, 3, 5, 6, 8, 9, 10, 11, 12\},$ $ \{1, 3, 6, 7, 8, 9, 10, 11, 12\},$ $ \{2, 3, 4, 5, 6, 7, 8, 11, 12\},$ $ \{2, 3, 4, 5, 6, 8, 10, 11, 12\},$ $ \{2, 4, 5, 6, 7, 8, 9, 11, 12\},$ $ \{2, 4, 5, 6, 8, 9, 10, 11, 12\},$
$ \{3, 4, 5, 6, 7, 8, 10, 11, 12\}\}$ }

~

{\footnotesize
$\mathcal{A}_5=
\{\{1, 2, 3, 4, 6\}, \{2, 3, 4, 5, 9\}, \{1, 4, 6, 7, 10\}, \{1, 5, 6, 7, 10\}, \{2, 5, 7, 9, 10\},$ $ \{1, 2, 3, 4, 5, 7, 8, 10, 11, 12\},$ $ \{1, 2, 3, 4, 7, 8, 9, 10, 11, 12\},$ $ \{1, 2, 3, 5, 6, 7, 8, 9, 11, 12\},$ $ \{1, 2, 3, 5, 6, 8, 9, 10, 11, 12\}, $ $\{1, 2, 3, 6, 7, 8, 9, 10, 11, 12\},$ $ \{1, 2, 4, 5, 6, 7, 8, 9, 11, 12\}, $ $\{1, 2, 4, 5, 6, 8, 9, 10, 11, 12\}, $ $\{1, 3, 4, 5, 6, 7, 8, 9, 11, 12\},$ $ \{1, 3, 4, 5, 6, 8, 9, 10, 11, 12\},$ $ \{1, 3, 4, 5, 7, 8, 9, 10, 11, 12\},$ $ \{2, 3, 4, 5, 6, 7, 8, 10, 11, 12\},$ $ \{2, 3, 4, 6, 7, 8, 9, 10, 11, 12\},$ $ \{3, 4, 5, 6, 7, 8, 9, 10, 11, 12\}\}$}

~

{\footnotesize
$\mathcal{A}_6=
\{\{2, 3, 4, 5, 6, 7, 8, 9, 10, 11, 12\}, \{1, 3, 4, 5, 6, 7, 8, 9, 10, 11, 12\}, \{1, 2, 4, 5, 6, 7, 8, 9, 10, 11, 12\},$ $ \{1, 2, 3, 5, 6, 7, 8, 9, 10, 11, 12\}, $ $\{1, 2, 3, 4, 6, 7, 8, 9, 10, 11, 12\},$ $ \{1, 2, 3, 4, 5, 7, 8, 9, 10, 11, 12\},$ $ \{1, 2, 3, 4, 5, 6\}\}$}

~

{\footnotesize
$\mathcal{A}_7=\{\{1,2,3,4,5,6,7,8,9,10,11,12\}\}$}

~

Total number of sets: 108

\section{Cross-Sperner families}\label{app:cross_sperner}

A construction for $(n,k) = (7,3)$ with value 6480:

\noindent$\mathcal{F}_1 = \{12,$ $125,$ $126,$ $1256,$ $127,$ $57,$ $157,$ $257,$ $1257,$ $357,$ $457,$ $3457,$ $1267,$ $567,$ $1567,$ $2567,$ $12567,$ $3567,$ $4567,$ $34567\}$

\noindent$\mathcal{F}_2 = \{23,$ $24,$ $234,$ $235,$ $245,$ $2345,$ $236,$ $246,$ $2346,$ $2356,$ $2456,$ $23456,$ $237,$ $247,$ $2347,$ $2367,$ $2467,$ $23467\}$

\noindent$\mathcal{F}_3 = \{13,$ $14,$ $134,$ $135,$ $145,$ $1345,$ $136,$ $146,$ $1346,$ $1356,$ $1456,$ $13456,$ $137,$ $147,$ $1347,$ $1367,$ $1467,$ $13467\}$

\hrulefill

A construction for $(n,k) = (8,3)$ with value 51840:

\noindent$\mathcal{F}_1 = \{25,$ $125,$ $235,$ $1235,$ $245,$ $256,$ $1256,$ $2356,$ $12356,$ $2456,$ $257,$ $1257,$ $2357,$ $12357,$ $2457,$ $2567,$ $12567,$ $23567,$ $123567,$ $24567,$ $48,$ $248,$ $258,$ $458,$ $2458,$ $468,$ $2468,$ $2568,$ $4568,$ $24568,$ $478,$ $2478,$ $2578,$ $4578,$ $24578,$ $4678,$ $24678,$ $25678,$ $45678,$ $245678\}$

\noindent$\mathcal{F}_2 = \{14,$ $124,$ $34,$ $134,$ $234,$ $1234,$ $145,$ $345,$ $1345,$ $146,$ $1246,$ $346,$ $1346,$ $2346,$ $12346,$ $1456,$ $3456,$ $13456,$ $147,$ $1247,$ $347,$ $1347,$ $2347,$ $12347,$ $1457,$ $3457,$ $13457,$ $1467,$ $12467,$ $3467,$ $13467,$ $23467,$ $123467,$ $14567,$ $34567,$ $134567\}$

\noindent$\mathcal{F}_3 = \{18,$ $128,$ $38,$ $138,$ $238,$ $1238,$ $158,$ $358,$ $1358,$ $168,$ $1268,$ $368,$ $1368,$ $2368,$ $12368,$ $1568,$ $3568,$ $13568,$ $178,$ $1278,$ $378,$ $1378,$ $2378,$ $12378,$ $1578,$ $3578,$ $13578,$ $1678,$ $12678,$ $3678,$ $13678,$ $23678,$ $123678,$ $15678,$ $35678,$ $135678\}$

\hrulefill

A construction for $(n,k) = (6,4)$ with value 1764:

\noindent$\mathcal{F}_1 = \{345,$ $46,$ $146,$ $346,$ $456,$ $3456\}$

\noindent$\mathcal{F}_2 = \{13,$ $123,$ $134,$ $135,$ $136,$ $156,$ $1356\}$

\noindent$\mathcal{F}_3 = \{24,$ $124,$ $234,$ $125,$ $145,$ $245,$ $1245\}$

\noindent$\mathcal{F}_4 = \{235,$ $26,$ $126,$ $236,$ $256,$ $2356\}$

\hrulefill

A construction for $(n,k) = (7,4)$ with value 28350:

\noindent$\mathcal{F}_1 = \{13,$ $123,$ $14,$ $124,$ $134,$ $1234,$ $135,$ $145,$ $1345,$ $137,$ $147,$ $1347,$ $1357,$ $1457,$ $13457\}$

\noindent$\mathcal{F}_2 = \{125,$ $16,$ $126,$ $156,$ $256,$ $1256,$ $127,$ $1257,$ $167,$ $267,$ $1267,$ $1567,$ $2567,$ $12567\}$

\noindent$\mathcal{F}_3 = \{235,$ $245,$ $2345,$ $237,$ $247,$ $2347,$ $2357,$ $2457,$ $23457\}$

\noindent$\mathcal{F}_4 = \{36,$ $236,$ $46,$ $246,$ $346,$ $2346,$ $356,$ $456,$ $3456,$ $367,$ $467,$ $3467,$ $3567,$ $4567,$ $34567\}$

\section{A double cover of $\{0,1,2\}^5$ with 41 boxes}\label{app:boxes}

\begin{multicols}{2}

\noindent$ B_{1} = \{0\}\times \{1\}\times \{0,1\}\times \{0,1\}\times \{2\}$\\
$ B_{2} = \{0\}\times \{1\}\times \{0,1\}\times \{0,2\}\times \{2\}$\\
$ B_{3} = \{0\}\times \{1\}\times \{0,2\}\times \{1,2\}\times \{2\}$\\
$ B_{4} = \{0\}\times \{1\}\times \{1,2\}\times \{1,2\}\times \{2\}$\\
$ B_{5} = \{1\}\times \{0\}\times \{2\}\times \{0\}\times \{0,1\}$\\
$ B_{6} = \{1\}\times \{0\}\times \{2\}\times \{0\}\times \{0,2\}$\\
$ B_{7} = \{1\}\times \{0\}\times \{2\}\times \{0\}\times \{1,2\}$\\
$ B_{8} = \{1\}\times \{0\}\times \{0,1\}\times \{0,1\}\times \{0\}$\\
$ B_{9} = \{1\}\times \{0\}\times \{0,1\}\times \{0,2\}\times \{0\}$\\
$ B_{10} = \{1\}\times \{0\}\times \{0,2\}\times \{1,2\}\times \{0\}$\\
$ B_{11} = \{1\}\times \{0\}\times \{1,2\}\times \{1,2\}\times \{0\}$\\
$ B_{12} = \{2\}\times \{2\}\times \{0,1\}\times \{0,1\}\times \{1\}$\\
$ B_{13} = \{2\}\times \{2\}\times \{0,1\}\times \{0,2\}\times \{1\}$\\
$ B_{14} = \{2\}\times \{2\}\times \{0,2\}\times \{1,2\}\times \{1\}$\\
$ B_{15} = \{2\}\times \{2\}\times \{1,2\}\times \{1,2\}\times \{1\}$\\
$ B_{16} = \{0,1\}\times \{0,2\}\times \{1\}\times \{1,2\}\times \{1,2\}$\\
$ B_{17} = \{0,1\}\times \{0,2\}\times \{0,1\}\times \{0\}\times \{1,2\}$\\
$ B_{18} = \{0,1\}\times \{0,2\}\times \{0,2\}\times \{1,2\}\times \{1,2\}$\\
$ B_{19} = \{0,1\}\times \{1,2\}\times \{2\}\times \{0\}\times \{2\}$\\
$ B_{20} = \{0,1\}\times \{1,2\}\times \{2\}\times \{2\}\times \{0,1\}$\\
$ B_{21} = \{0,1\}\times \{1,2\}\times \{2\}\times \{0,1\}\times \{0,1\}$\\
\noindent$ B_{22} = \{0,1\}\times \{1,2\}\times \{0,1\}\times \{2\}\times \{0,1\}$\\
$ B_{23} = \{0,1\}\times \{1,2\}\times \{0,1\}\times \{0,1\}\times \{0,1\}$\\
$ B_{24} = \{0,2\}\times \{0,1\}\times \{0\}\times \{1\}\times \{0,1\}$\\
$ B_{25} = \{0,2\}\times \{0,1\}\times \{0\}\times \{0,2\}\times \{0,1\}$\\
$ B_{26} = \{0,2\}\times \{0,1\}\times \{2\}\times \{0\}\times \{2\}$\\
$ B_{27} = \{0,2\}\times \{0,1\}\times \{1,2\}\times \{0\}\times \{0,1\}$\\
$ B_{28} = \{0,2\}\times \{0,1\}\times \{1,2\}\times \{1,2\}\times \{0,1\}$\\
$ B_{29} = \{0,2\}\times \{0,2\}\times \{0\}\times \{1\}\times \{0,2\}$\\
$ B_{30} = \{0,2\}\times \{0,2\}\times \{0\}\times \{0,2\}\times \{0,2\}$\\
$ B_{31} = \{0,2\}\times \{0,2\}\times \{2\}\times \{0\}\times \{1\}$\\
$ B_{32} = \{0,2\}\times \{0,2\}\times \{1,2\}\times \{1\}\times \{0,2\}$\\
$ B_{33} = \{0,2\}\times \{0,2\}\times \{1,2\}\times \{0,2\}\times \{0,2\}$\\
$ B_{34} = \{1,2\}\times \{0,1\}\times \{1\}\times \{1,2\}\times \{1,2\}$\\
$ B_{35} = \{1,2\}\times \{0,1\}\times \{0,1\}\times \{0\}\times \{1,2\}$\\
$ B_{36} = \{1,2\}\times \{0,1\}\times \{0,2\}\times \{1,2\}\times \{1,2\}$\\
$ B_{37} = \{1,2\}\times \{1,2\}\times \{0\}\times \{0,2\}\times \{0,2\}$\\
$ B_{38} = \{1,2\}\times \{1,2\}\times \{1\}\times \{1\}\times \{0,2\}$\\
$ B_{39} = \{1,2\}\times \{1,2\}\times \{2\}\times \{0\}\times \{1\}$\\
$ B_{40} = \{1,2\}\times \{1,2\}\times \{0,2\}\times \{1\}\times \{0,2\}$\\
$ B_{41} = \{1,2\}\times \{1,2\}\times \{1,2\}\times \{0,2\}\times \{0,2\}$\\
\end{multicols}

\section{A spanning subgraph of the 6-cube with 81 edges and diameter 6}\label{app:erdos_hypercube}

The edges of the subgraph are as follows:

\begin{multicols}{3}
\noindent(000110, 000100) \\
(000110, 001110) \\
(000110, 010110) \\
(000111, 000011) \\
(000111, 100111) \\
(000100, 000101) \\
(000100, 000000) \\
(000100, 001100) \\
(000100, 010100) \\
(000100, 100100) \\
(000101, 000001) \\
(000101, 010101) \\
(000101, 100101) \\
(000010, 000000) \\
(000010, 100010) \\
(000011, 000001) \\
(000011, 001011) \\
(000000, 000001) \\
(000000, 001000) \\
(000000, 010000) \\
(000000, 100000) \\
(000001, 001001) \\
(000001, 010001) \\
(001110, 011110) \\
(001111, 001101) \\
(001111, 101111) \\
(001100, 011100) \\
(001100, 101100) \\
(001101, 001001) \\
(001010, 001000) \\
(001010, 011010) \\
(001011, 011011) \\
(001000, 011000) \\
(010110, 110110) \\
(010111, 010101) \\
(010111, 010011) \\
(010100, 110100) \\
(010010, 010000) \\
(010010, 110010) \\
(010011, 110011) \\
(010000, 110000) \\
(010001, 011001) \\
(011110, 111110) \\
(011111, 011101) \\
(011111, 011011) \\
(011111, 111111) \\
(011100, 111100) \\
(011101, 011001) \\
(011010, 111010) \\
(011000, 111000) \\
(100110, 100100) \\
(100110, 101110) \\
(100111, 110111) \\
(100101, 100001) \\
(100010, 101010) \\
(100011, 100001) \\
(100011, 101011) \\
(100000, 101000) \\
(101110, 111110) \\
(101111, 111111) \\
(101100, 101101) \\
(101101, 111101) \\
(101010, 101011) \\
(101011, 111011) \\
(101000, 101001) \\
(101001, 111001) \\
(110110, 110111) \\
(110111, 111111) \\
(110100, 110101) \\
(110101, 111101) \\
(110010, 110011) \\
(110011, 111011) \\
(110000, 110001) \\
(110001, 111001) \\
(111110, 111111) \\
(111110, 111100) \\
(111111, 111101) \\
(111111, 111011) \\
(111010, 111011) \\
(111010, 111000) \\
(111011, 111001) \\
\end{multicols}

\end{document}